\documentclass[11pt]{amsart}

\usepackage{a4wide}
\usepackage{amssymb}
\usepackage{mathrsfs,mathtools}
\usepackage{enumerate}
\usepackage{esint}
\usepackage{titletoc}
\usepackage{hyperref}
\usepackage{xcolor}

\theoremstyle{plain}
\newtheorem{theorem}{Theorem}[section]
\newtheorem{lemma}[theorem]{Lemma}
\newtheorem{corollary}[theorem]{Corollary}
\newtheorem{proposition}[theorem]{Proposition}

\theoremstyle{definition}
\newtheorem{definition}[theorem]{Definition}

\numberwithin{equation}{section}

\DeclareMathOperator{\supp}{supp}

\makeatletter
\@namedef{subjclassname@2020}{\textup{2020} Mathematics Subject Classification}
\makeatother

\title{Singularities of solutions of nonlocal nonlinear equations}

\author{Minhyun Kim}
\address{Department of Mathematics \& Research Institute for Natural Sciences, Hanyang University, 04763 Seoul, Republic of Korea}
\email{minhyun@hanyang.ac.kr}

\author{Se-Chan Lee}
\address{School of Mathematics, Korea Institute for Advanced Study, 02455 Seoul, Republic of Korea}
\email{sechan@kias.re.kr}

\thanks{M. Kim is supported by the National Research Foundation of Korea (NRF) grant funded by the Korean government (MSIT) (RS-2023-00252297). S.-C. Lee is supported by the KIAS Individual Grant (No. MG099001) at the Korea Institute for Advanced Study.}

\begin{document}

	\begin{abstract}
		We study the local behavior of weak solutions, with possible singularities, of nonlocal nonlinear equations. We first prove that sets of capacity zero are removable for weak solutions under certain integrability conditions. We then characterize the asymptotic behavior of singular solutions near an isolated singularity in terms of the fundamental solution.
	\end{abstract}
	
	\maketitle

	\section{Introduction}

	The classical results due to Serrin~\cite{Ser64, Ser65} deal with the local behavior of weak solutions of quasilinear equations in divergence form, whose model equation is the $p$-Laplace equation $-\Delta_p u=0$ for $p>1$. Regarding the regular behavior of solutions, he established the local boundedness, the weak Harnack inequality, the Harnack inequality and H\"older continuity by adapting Moser's iteration technique \cite{Mos60, Mos61} to the quasilinear setting. These regularity results were then employed to investigate the singularity of solutions, leading to the removable singularity theorem and the isolated singularity theorem. He first showed that a solution $u$ in $\Omega \setminus E$ can be extended as a solution of the same equation in the entire domain $\Omega$, provided that $E$ is negligible in the sense of capacity and that $u$ satisfies certain integrability conditions. Furthermore, he proved that singular solutions near an isolated singularity must behave like the fundamental solution of the $p$-Laplace equation.

	The goal of this paper is to establish a nonlocal counterpart of the results in celebrated papers \cite{Ser64, Ser65}. In particular, we are interested in the nature of removable singularities of solutions and the behavior of a solution near an isolated singularity in the nonlocal nonlinear framework. It is noteworthy that the local regularity of weak solutions of nonlocal equations has been relatively well studied in various works, including those by Caffarelli--Chan--Vasseur~\cite{CCV11} and Kassmann~\cite{Kas09} for linear equations and by Di Castro--Kuusi--Palatucci~\cite{DCKP14, DCKP16} and Cozzi~\cite{Coz17} for nonlinear equations.
	
	Let $n \in \mathbb{N}$, $0<s<1<p<\infty$ and $\Lambda \geq 1$, and assume that $sp \leq n$. Let $\Omega$ be a nonempty open set in $\mathbb{R}^n$, which may be unbounded. We are concerned with weak solutions $u$ of the nonlocal nonlinear equation
	\begin{equation}\label{eq-main}
		\mathcal{L}u+b(x, u)=0
	\end{equation}
	in $\Omega$ or in a subset of $\Omega$, where $\mathcal{L}$ is an integro-differential operator of the form
	\begin{equation*}
		\mathcal{L}u(x) \coloneqq 2\, \mathrm{p.v.} \int_{\mathbb{R}^n} |u(x)-u(y)|^{p-2} (u(x)-u(y)) k(x, y) \,\mathrm{d}y,
	\end{equation*}
	with a measurable kernel $k: \mathbb{R}^n \times \mathbb{R}^n \to [0, \infty]$ satisfying the symmetry $k(x, y)=k(y, x)$ and the ellipticity condition
	\begin{equation}\label{eq-ellipticity}
		\frac{\Lambda^{-1}}{|x-y|^{n+sp}} \leq k(x, y) \leq \frac{\Lambda}{|x-y|^{n+sp}},
	\end{equation}
    and $b: \Omega \times \mathbb{R} \to \mathbb{R}$ is a measurable function satisfying the structure condition
	\begin{equation}\label{eq-str}
		|b(x, z)| \leq b_1(x) |z|^{p-1} + b_2(x)
	\end{equation}
	with nonnegative functions $b_1, b_2 \in L^{n/(sp-\varepsilon)}(\Omega)$ ($0<\varepsilon<sp$). The principal model equation is the fractional $p$-Laplace equation $(-\Delta_p)^su=0$.
	
	Our first main theorem on the removable singularity reads as follows. See Section~\ref{sec-preliminaries} for the definitions of weak solutions and sets of capacity zero.
 
	\begin{theorem}\label{thm-rem}
		Let $E$ be a relatively closed set in $\Omega$ of $(t, q)$-capacity zero, where
		\begin{equation*}
			q\geq p \quad\text{and}\quad
			\begin{cases}
				t=s &\text{if}~q=p, \\
				t \in (s, 1) &\text{if}~q>p.
			\end{cases}
		\end{equation*}
		If $u$ is a weak solution of \eqref{eq-main} in $\Omega \setminus E$ such that $u$ is continuous in $\Omega \setminus E$ and $u \in L^{(1+\delta)\theta}(\Omega \setminus E)$ for some $\delta>0$, where
		\begin{equation*}
			\theta=q(p-1)/(q-p),
		\end{equation*}
		then $u$ has a representative that solves the same equation \eqref{eq-main} in $\Omega$ and is continuous in $\Omega$.
	\end{theorem}
	A few remarks on Theorem~\ref{thm-rem} are in order. In contrast to the local case, $u$ is a priori defined on the whole of $\mathbb{R}^n$ in Theorem~\ref{thm-rem}. Moreover, since $E$ is the set of measure zero by Lemma~\ref{lem-measurezero}, the condition $u \in L^{(1+\varepsilon)\theta}(\Omega \setminus E)$ in Theorem~\ref{thm-rem} is equivalent to $u \in L^{(1+\varepsilon)\theta}(\Omega)$. When $q=p$, this condition is understood as $u \in L^\infty(\Omega)$. Note, however, that the requirement $u \in W^{s, p}_{\mathrm{loc}}(\Omega\setminus E)\cap L^{p-1}_{sp}(\mathbb{R}^n)$ for $u$ to be a weak solution cannot be replaced by $u \in W^{s, p}_{\mathrm{loc}}(\Omega)\cap L^{p-1}_{sp}(\mathbb{R}^n)$ (see Definition~\ref{def-sol}). This makes the problem significantly more challenging.

	If $E$ is a singleton, then $E$ is of $(s, p)$-capacity zero according to Lemma~\ref{lem-cap-point}. Thus, Theorem~\ref{thm-rem} implies the following removable isolated singularity theorem under a suitable growth condition in the subcritical case.
	
	\begin{corollary}\label{cor-rem}
		Assume that $sp<n$ and that $0 \in \Omega$. Let $u$ be a weak solution of \eqref{eq-main} in $\Omega \setminus \{0\}$ such that
		\begin{equation*}
			u(x)=O\left(|x|^{-(n-sp)/(p-1)+\delta} \right)
		\end{equation*}
		near the origin for some $\delta>0$. Then $u$ has a representative that solves the same equation \eqref{eq-main} in $\Omega$ and is continuous in $\Omega$.
	\end{corollary}

Let us mention that Del Pezzo--Quaas~\cite{DPQ23} recently verified that the fundamental solution $\Gamma_{s, p}$ for $(-\Delta_p)^s$ is given by 
	\begin{equation*}
		\Gamma_{s, p}(x)=
		\begin{cases}
			|x|^{-(n-sp)/(p-1)} & \text{if $sp<n$},\\
			-\log|x| & \text{if $sp=n$},
		\end{cases}
	\end{equation*}
up to multiplication constants depending only on $n$, $s$ and $p$. Thus, Corollary~\ref{cor-rem} states that $u$ has a removable singularity at the origin if it blows up slower than the fundamental solution. 

Our second main theorem is concerned with the classification of the behavior of solutions near an isolated singularity. We begin with the subcritical case. Here and below, $u \eqsim v$ means that $C_1 \leq u/v \leq C_2$ for some positive constants $C_1$ and $C_2$. ($u \lesssim v$ and $u \gtrsim v$ can be defined in the same way.)
	\begin{theorem}\label{thm-iso-sub}
		Assume that $sp<n$ and that $0 \in \Omega$. Let $u$ be a weak solution of \eqref{eq-main} in $\Omega \setminus \{0\}$ such that $u$ is continuous and bounded from below in $\Omega \setminus \{0\}$. Then either $u$ has a removable singularity at the origin or
		\begin{equation*}
			u \eqsim |x|^{-(n-sp)/(p-1)}
		\end{equation*}
		in a neighborhood of the origin.
	\end{theorem}

	As an immediate consequence of Theorem~\ref{thm-iso-sub}, an alternative version of Corollary~\ref{cor-rem} is obtained in the subcritical case as follows. We point out that Corollary~\ref{cor-rem2} improves Corollary~\ref{cor-rem} in terms of the growth condition, but it requires an additional boundedness assumption from one side to take advantage of Theorem~\ref{thm-iso-sub}.

\begin{corollary}\label{cor-rem2}
Assume that $sp<n$ and that $0 \in \Omega$. Let $u$ be a weak solution of \eqref{eq-main} in $\Omega \setminus \{0\}$ such that $u$ is bounded from one side in $\Omega \setminus \{0\}$ and
	\begin{equation*}
		u(x)=o\left(|x|^{-(n-sp)/(p-1)} \right)
	\end{equation*}
	near the origin. Then $u$ has a representative that solves the same equation \eqref{eq-main} in $\Omega$ and is continuous in $\Omega$.
\end{corollary}

	In the critical case $sp=n$, we are able to achieve the partial classification result in terms of the fundamental solution $\Gamma_{s, p}$.
	
	\begin{theorem}\label{thm-iso-log}
		Assume that $sp=n$ and that $0 \in \Omega$. Let $u$ be a weak solution of \eqref{eq-main} in $\Omega \setminus \{0\}$ such that $u$ is continuous and bounded from below in $\Omega \setminus \{0\}$. Then either $u$ has a removable singularity at the origin or
		\begin{equation*}
			u \lesssim -\log |x|
		\end{equation*}
		in a neighborhood of the origin.
	\end{theorem}

	Let us briefly summarize the historical background. Removable singularity and isolated singularity of solutions of \textit{local} equations have been of great interest; just to name a few, we refer to \cite{Boc03, Car67, GS55} for linear equations, \cite{NSS03, Ser64, Ser65} for quasilinear equations and \cite{ASS11, Lab00, Lab01} for fully nonlinear equations. All results on isolated singularities share the essence that if $u$ is a solution in the punctured ball, then either $u$ has a removable singularity or $u$ behaves like a fundamental solution of the associated local operator. We also refer to \cite{CGS89, CL99, KMPS99} for the scalar curvature equations and \cite{Avi83, BVGHV19, BVV91, GS81, Lio80} for other semilinear equations, where the asymptotic behavior is captured by functions different from fundamental solutions. See also comprehensive monographs \cite{QS19, Ver96} and references therein. 

    There has been a growing interest in the study on singularities of solutions of \emph{nonlocal} equations. Li--Wu--Xu~\cite{LWX18}, Li--Liu--Wu--Xu~\cite{LLWX20} and Klimsiak~\cite{Kli25} established the B\^ocher-type theorems for nonnegative solutions of fractional linear equations, which characterize isolated singularities by Dirac mass $\delta_0$ in the distribution sense. Moreover, there have been several works focusing on fractional semilinear equations of the form
    \begin{equation}\label{eq-semilinear}
        (-\Delta)^s u = u^\alpha
    \end{equation}
    for some $\alpha>1$. Caffarelli--Jin--Sire--Xiong~\cite{CJSX14}, Chen--Quaas~\cite{CQ18}, Chen--V\'eron~\cite{CV23}, Chen--Li--Ou~\cite{CLO05} and Yang--Zou~\cite{YZ19,YZ21} investigated the local behavior of solutions of \eqref{eq-semilinear} in $B_1 \setminus \{0\}$ with zero complement data for different ranges of $\alpha$. In particular, for the critical case $\alpha=(n+2s)/(n-2s)$ that corresponds to the fractional Yamabe equation, Jin--de Queiroz--Sire--Xiong~\cite{JdQSX17} described local behavior of solutions with a singular set of fractional capacity zero. Let us also mention the work by Ao--Gonz\'alez--Hyder--Wei~\cite{AGHW22}, where a removability result for the equation \eqref{eq-semilinear} was established. 
    
    We point out that most of these papers rely heavily on the Caffarelli--Silvestre extension~\cite{CS07} that guarantees the validity of various tools in the local setting. In fact, the condition that $u$ vanishes on the complement of $B_1$ removes the effect of the nonlocal tail term, which encodes the long-range interaction of $u$.

The novelty of the present paper lies in preserving the nonlocal character of the problem. We impose no restrictions on the behavior of solutions in the complement of $\Omega$ nor on the structure of $\mathcal{L}$. To the best of our knowledge, our results are new even in the linear case, i.e.\ $p=2$. Furthermore, despite the technical difficulties arising from $b$, we take this lower-order term $b$ into consideration to potentially extend our results to the semilinear context. In particular, one may further investigate the local behavior of positive singular solutions of nonlocal nonlinear equations of the form $\mathcal{L}u=u^\alpha$ with $1<\alpha < n(p-1)/(n-sp)$, similar to how \cite{BVGHV19, GS81} employed the results of \cite{Ser64, Ser65} in the local semilinear setting. 
	
Let us finally illustrate the strategies to deduce the principal results of the present paper. In contrast to the simple display of our main theorems, their proofs are quite intricate and rely heavily on advanced tools from nonlocal regularity theory. As mentioned earlier, the main difficulty in proving the removability theorem (Theorem~\ref{thm-rem}) arises from the lack of the regularity of $u$ in the whole domain $\Omega$. To address this difficulty, we develop the De Giorgi--Nash--Moser theory tailored to our situation. Indeed, in Section~\ref{sec-regularity}, we obtain Moser-type Caccioppoli estimates that contain a cut-off function $\bar{\eta}$ vanishing in a neighborhood of $E$, which is introduced to avoid the possible singularity on $E$, and some powers of weak (sub-)solutions $u$. On the one hand, we need a polynomial of $u$ with a certain growth rate to use the integrability condition given in Theorem~\ref{thm-rem}. On the other hand, we also need a polynomial of $u$ with an arbitrarily large power in order for Moser's iteration to work. By taking a test function as the minimum of these two polynomials, we arrive at the desired Caccioppoli estimate.

To prove Theorem~\ref{thm-rem} by means of Caccioppoli-type estimates, we now choose $\bar{\eta}=\bar{\eta}_j$ as a sequence of functions approximating the capacity potential of $E$ in $\mathbb{R}^n$. Since $E$ is of capacity zero and $u$ satisfies the integrability condition, taking the limit as $j \to \infty$ eliminates the effect of $E$ (or $\bar{\eta}$) from the integral inequality. Moreover, another limiting argument that will be specified in Section~\ref{sec-removable} guarantees that the test function can be reduced to a single polynomial with an arbitrarily large power. We conclude by using the standard Moser's iteration method that $u$ belongs to the desired function space and so the set $E$ is indeed removable.
 
 The proof of the isolated singularity theorem (Theorem~\ref{thm-iso-sub} or \ref{thm-iso-log}) is carried out by combining several steps as follows. If a weak solution $u$ has a non-removable isolated singularity at $0$, then
	\begin{enumerate}[(i)]
		\item $\lim_{x \to 0}u(x)=\infty$: In the local setting, Lemma~\ref{lem-non-rem} follows from a relatively simple argument utilizing Theorem~\ref{thm-rem}, the Harnack inequality and the maximum principle. More precisely, the maximum principle plays a crucial role in verifying the desired limiting behavior of $u$ in the region between two particular spheres on which $u$ becomes very large. However, the maximum principle for nonlocal operators provides weaker information compared to the local case, as it depends not only on the boundary data but also on the complement data. To overcome this challenge, we replace the maximum principle by the weak Harnack inequality (Lemma~\ref{lem-WHI}), which encodes the global behavior of $u$ into the (finite) nonlocal tail term. In particular, we observe that a positive weak solution $u$ of \eqref{eq-main} in $\Omega \setminus \{0\}$ satisfies De Giorgi-type Caccioppoli estimates in $\Omega$, while $u$ itself may not belong to $W^{s, p}_{\mathrm{loc}}(\Omega)$. 
  
		\item $u \lesssim \Gamma_{s,p}$ near $0$: We begin with introducing an invariant quantity $\overline{K}\coloneqq \mathcal{E}(u, \varphi)+\int b(\varphi-1)$, where $\varphi$ is contained in a suitable function class, see Lemma~\ref{lem-const} for the precise description. Then we derive a weaker version of upper bounds for singular solutions in Lemma~\ref{lem-upper} based on a new localization technique. This technique reveals the relationship between the invariant value $\overline{K}$ and the capacity of a specific annulus, by clarifying the effect of the nonlocal tail term. Furthermore, we end up with the complete upper bound estimates after solving the difference inequality for the integral of $b$ in Lemma~\ref{lem-integrable}. We note that the desired upper bound follows from the fact that the capacity of an annulus can be captured by the growth rate of $\Gamma_{s,p}$ in view of Lemma~\ref{lem-cap-ball}.

         \item $u \gtrsim \Gamma_{s,p}$ near $0$: In Lemma~\ref{lem-K}, we define another invariant value $K\coloneqq \mathcal{E}(u, \varphi)+\int b \varphi$ for $\varphi$ in the same function class as before. By developing similar estimates on $K$ but in a more careful way together with the results from step (ii), it turns out that $K$ has a positive sign in Lemma~\ref{lem-positive-K}. The lower bound of singular solutions is obtained in Lemma~\ref{lem-lower}  by using the positivity of $K$ and suggesting an appropriate test function in terms of the $(-\Delta_p)^s$-potential. Even though an additional difficulty arises from the maximum principle unlike the local situation, the upper bound of $u$ allows us to restrict our attention to the boundary behavior of $u$ near spheres. A similar procedure in step (iii) seems to fail in the critical case $sp=n$, where we could only obtain the upper bounds.
	\end{enumerate}

	The paper is organized as follows. In Section~\ref{sec-preliminaries}, we summarize the notation and preliminary results on functional inequalities, weak solutions and capacities. Section~\ref{sec-regularity} consists of local regularity estimates, namely, various Caccioppoli-type estimates, H\"older regularity and the Harnack inequality. The proofs for Theorems~\ref{thm-rem} and \ref{thm-iso-sub} are presented in Sections~\ref{sec-removable} and \ref{sec-isolated}, respectively.

\section{Preliminaries}\label{sec-preliminaries}

\emph{Throughout the paper, we assume that $0<s<1<p<\infty$ and $n \in \mathbb{N}$, that $k$ is a symmetric measurable kernel satisfying \eqref{eq-ellipticity} with $\Lambda \geq 1$ and that $\Omega$ is a nonempty open subset of $\mathbb{R}^n$. (In Section~\ref{sec-isolated}, we further assume that $\Omega$ is bounded and contains the origin.) Moreover, whenever the equation \eqref{eq-main} is concerned, we also assume that $sp\leq n$ and that $b$ satisfies \eqref{eq-str} with nonnegative functions $b_1, b_2 \in L^{n/(sp-\varepsilon)}(\Omega)$ for some $\varepsilon \in (0, sp)$.}

\subsection{Functional inequalities}

In this section, we recall definitions of several function spaces and provide inequalities related to these spaces. The fractional Sobolev space $W^{s, p}(\Omega)$ consists of measurable functions $u: \Omega \to [-\infty, \infty]$ whose fractional Sobolev norm
\begin{align*}
\|u\|_{W^{s, p}(\Omega)}
&\coloneqq \left( \|u\|_{L^p(\Omega)}^p + [u]_{W^{s, p}(\Omega)}^p \right)^{1/p} \\
&\coloneqq \left( \int_\Omega |u(x)|^p \,\mathrm{d}x + \int_\Omega \int_\Omega \frac{|u(x)-u(y)|^p}{|x-y|^{n+sp}} \,\mathrm{d}y\,\mathrm{d}x \right)^{1/p}
\end{align*}
is finite. By $W^{s, p}_{\mathrm{loc}}(\Omega)$ we denote the space of functions $u$ such that $u \in W^{s, p}(G)$ for every open $G \Subset \Omega$.

Since we study nonlocal equations, the \emph{tail space}
\begin{equation*}
L^{p-1}_{sp}(\mathbb{R}^n) \coloneqq \left\{ u~\text{measurable}: \int_{\mathbb{R}^n} \frac{|u(y)|^{p-1}}{(1+|y|)^{n+sp}} \,\mathrm{d}y < \infty \right\}
\end{equation*}
is useful. Note that $u \in L^{p-1}_{sp}(\mathbb{R}^n)$ if and only if the \emph{nonlocal tail} (or \emph{tail} for short)
\begin{equation*}
\mathrm{Tail}(u; x_0, r) \coloneqq \left( r^{sp} \int_{\mathbb{R}^n \setminus B_r(x_0)} \frac{|u(y)|^{p-1}}{|y-x_0|^{n+sp}} \,\mathrm{d}y \right)^{1/(p-1)}
\end{equation*}
is finite for any $x_0 \in \mathbb{R}^n$ and $r>0$.

Some functional inequalities, such as the fractional Sobolev inequality and H\"older's inequality, are presented here. We begin with fractional Sobolev inequalities. See Di Nezza--Palatucci--Valdinoci~\cite[Theorems~6.5~and~6.7]{DNPV12} for their proofs.

\begin{theorem}[Fractional Sobolev inequality]
\label{thm-Sobolev}
Assume that $sp<n$ and let $p^{\ast}_s = np/(n-sp)$. Let $B=B_r(x_0)$ be a ball. Then there exists $C=C(n, s, p)>0$ such that
\begin{equation*}
\|u\|_{L^{p^{\ast}_s}(B)} \leq C \left( [u]_{W^{s, p}(B)} + r^{-s} \|u\|_{L^p(B)} \right)
\end{equation*}
for any $u \in W^{s, p}(B)$ and that
\begin{equation*}
\|u\|_{L^{p^{\ast}_s}(B)} \leq C [u]_{W^{s, p}(\mathbb{R}^n)}
\end{equation*}
for any measurable and compactly supported function $u:\mathbb{R}^n \to \mathbb{R}$.
\end{theorem}

Note that Theorem~\ref{thm-Sobolev} provides the embedding $W^{s, p}(B) \subset L^{p^{\ast}_s}(B)$. Another useful embedding for fractional Sobolev spaces is $W^{s, p}(B) \subset W^{\tilde{s}, \tilde{p}}(B)$, where $\tilde{s} \in (0,s)$ and $\tilde{p} \in (1, p)$, which is obtained from the following lemma and the embedding $L^p(B) \subset L^{\tilde{p}}(B)$.

\begin{lemma}\label{lem-Jensen}
\textup{(Cozzi~\cite[Lemma~4.6]{Coz17})}
Let $\tilde{s} \in (0,s)$, $\tilde{p} \in (1, p)$ and $B=B_r(x_0)$. Then there exists a constant $C=C(n, s, \tilde{s}, p, \tilde{p}) > 0$ such that
\begin{equation*}
[u]_{W^{\tilde{s}, \tilde{p}}(B)} \leq C r^{n\frac{p-\tilde{p}}{p\tilde{p}}+s-\tilde{s}} [u]_{W^{s, p}(B)}
\end{equation*}
for any $u \in W^{s, p}(B)$.
\end{lemma}

We also need the following lemma, which is a generalized version of Lemma~\ref{lem-Jensen}.

\begin{lemma}\label{lem-Holder}
Let $\tilde{s} \in (0,s)$, $\tilde{p} \in (1, p)$, and let $G$ be a bounded open set. Then there exists a constant $C=C(n, s, \tilde{s}, p, \tilde{p}) > 0$ such that
\begin{equation}\label{eq-Holder}
\int_G \int_G u(x) \frac{|v(x)-v(y)|^{\tilde{p}}}{|x-y|^{n+\tilde{s}\tilde{p}}} \,\mathrm{d}y\,\mathrm{d}x \leq C (\mathrm{diam}\,G)^{(s-\tilde{s})\tilde{p}} \|u\|_{L^{p/(p-\tilde{p})}(G)} [v]_{W^{s, p}(G)}^{\tilde{p}}
\end{equation}
for any $u \in L^{p/(p-\tilde{p})}(G)$ and $v \in W^{s, p}(G)$.
\end{lemma}

Note that Lemma~\ref{lem-Holder} is also true when $\tilde{s}=s$ and $\tilde{p}=p$.

\begin{proof}
Let $L$ denote the left-hand side of \eqref{eq-Holder}. It follows from H\"older's inequality that
\begin{equation*}
L \leq \|u\|_{L^{p/(p-\tilde{p})}(G)} \left( \int_G \left( \int_G \frac{|v(x)-v(y)|^{\tilde{p}}}{|x-y|^{n+\tilde{s}\tilde{p}}} \,\mathrm{d}y \right)^{p/\tilde{p}} \mathrm{d}x \right)^{\tilde{p}/p}.
\end{equation*}
Let $\mu(x, \mathrm{d}y) = |x-y|^{-n+(s-\tilde{s})\tilde{p}} \,\mathrm{d}y$ and $d=\mathrm{diam}\,G$. Then
\begin{equation*}
\mu(x, G) \leq \int_{B_{d}(x)} |x-y|^{-n+(s-\tilde{s})\tilde{p}} \,\mathrm{d}y \leq C d^{(s-\tilde{s})\tilde{p}}<\infty
\end{equation*}
for any $x \in G$. Thus, an application of Jensen's inequality shows that
\begin{align*}
L
&\leq \|u\|_{L^{p/(p-\tilde{p})}(G)} \left( \int_G \mu(x, G)^{p/\tilde{p}-1} \int_G \frac{|v(x)-v(y)|^{p}}{|x-y|^{sp}} \mu(x, \mathrm{d}y) \,\mathrm{d}x \right)^{\tilde{p}/p} \\
&\leq C d^{(s-\tilde{s})\tilde{p}}\|u\|_{L^{p/(p-\tilde{p})}(G)} [v]_{W^{s, p}(G)}^{\tilde{p}},
\end{align*}
where $C=C(n, s, \tilde{s}, p, \tilde{p})>0$.
\end{proof}

We close the section with the following inequalities, which appear several times throughout the paper.

\begin{lemma}\label{lem-tail}
Let $S, G \subset \mathbb{R}^n$ be nonempty bounded measurable sets such that $S \Subset \mathrm{int}\, G$, and let $d= \mathrm{dist}(S, G^c)$. Then there exists a constant $C=C(n, s, p)>0$ such that
\begin{equation}\label{eq-tail-x}
\int_S \int_{G^c} \frac{|u(x)|^{p-1}}{|x-y|^{n+sp}} \,\mathrm{d}y \,\mathrm{d}x \leq Cd^{-sp} \|u\|_{L^{p-1}(S)}^{p-1}
\end{equation}
for any $u \in L^{p-1}(S)$. Moreover, if $x_0 \in S$, then
\begin{equation*}
\int_S \int_{G^c} \frac{|u(y)|^{p-1}}{|x-y|^{n+sp}} \,\mathrm{d}y \,\mathrm{d}x \leq \frac{|S|}{d^{sp}} \left( 1+\frac{\mathrm{diam}\,S}{d} \right)^{n+sp} \mathrm{Tail}^{p-1}(u; x_0, d)
\end{equation*}
for any $u \in L^{p-1}_{sp}(\mathbb{R}^n)$.
\end{lemma}

\begin{proof}
Since for any $x \in S$
\begin{equation*}
\int_{G^c} |x-y|^{-n-sp} \,\mathrm{d}y \leq \int_{\mathbb{R}^n \setminus B_d(x)} |x-y|^{-n-sp} \,\mathrm{d}y \leq C d^{-sp},
\end{equation*}
the inequality \eqref{eq-tail-x} follows.

If $x \in S$ and $y \in G^c$, then $|y-x_0| \leq |y-x|+|x-x_0| \leq (1+\frac{\mathrm{diam}\,S}{d})|x-y|$, and hence
\begin{align*}
\int_S \int_{G^c} \frac{|u(y)|^{p-1}}{|x-y|^{n+sp}} \,\mathrm{d}y \,\mathrm{d}x
&\leq \left( 1+\frac{\mathrm{diam}\,S}{d} \right)^{n+sp} \int_S \int_{\mathbb{R}^n \setminus B_d(x_0)} \frac{|u(y)|^{p-1}}{|y-x_0|^{n+sp}} \,\mathrm{d}y \,\mathrm{d}x \\
&= \frac{|S|}{d^{sp}} \left( 1+\frac{\mathrm{diam}\,S}{d} \right)^{n+sp} \mathrm{Tail}^{p-1}(u; x_0, d),
\end{align*}
as desired.
\end{proof}

\subsection{Weak solution}\label{sec-solution}

In this section, we define a weak solution of the equation \eqref{eq-main} and study some properties of weak solutions. For this purpose, we define for measurable functions $u, v: \mathbb{R}^n \to [-\infty, \infty]$ a quantity
\begin{equation*}
\mathcal{E}(u,v)\coloneqq\int_{\mathbb{R}^n} \int_{\mathbb{R}^n} |u(x)-u(y)|^{p-2}(u(x)-u(y))(v(x)-v(y)) k(x, y) \,\mathrm{d}y\,\mathrm{d}x,
\end{equation*}
provided that it is finite.

\begin{definition}\label{def-sol}
A function $u \in W^{s, p}_{\mathrm{loc}}(\Omega) \cap L^{p-1}_{sp}(\mathbb{R}^n)$ is a {\it weak supersolution} (resp.\ \emph{weak subsolution}) of \eqref{eq-main} in $\Omega$ if
\begin{equation}\label{eq-sol}
\mathcal{E}(u, \varphi) + \int_{\Omega}b(x, u(x))\varphi(x)\,\mathrm{d}x \geq 0 ~\left(\text{resp.} \leq 0 \right)
\end{equation}
for all nonnegative $\varphi \in C^\infty_c(\Omega)$. A function $u \in W^{s, p}_{\mathrm{loc}}(\Omega) \cap L^{p-1}_{sp}(\mathbb{R}^n)$ is a {\it weak solution} of \eqref{eq-main} in $\Omega$ if \eqref{eq-sol} holds for all $\varphi \in C^\infty_c(\Omega)$.
\end{definition}

As test functions are often not in $C^\infty_c(\Omega)$, it is desirable to consider a larger class of functions than $C^\infty_c(\Omega)$ in Definition~\ref{def-sol}.

\begin{proposition}\label{prop-test}
A function $u\in W^{s, p}_{\mathrm{loc}}(\Omega) \cap L^{p-1}_{sp}(\mathbb{R}^n)$  is a weak supersolution (resp.\ weak subsolution) of \eqref{eq-main} in $\Omega$ if and only if \eqref{eq-sol} holds for all nonnegative $\varphi \in W^{s, p}_{\mathrm{loc}}(\Omega)$ with $\supp{\varphi}\Subset \Omega$.
\end{proposition}

Note that similar results for weak subsolutions and weak solutions also hold.

\begin{proof}
Let us fix an open set $G$ such that $\supp{\varphi} \Subset G \Subset \Omega$. By mollification, there exist nonnegative functions $\varphi_j \in C^\infty_c(G)$ such that $\varphi_j \to \varphi$ in $W^{s, p}(G)$ as $j \to \infty$. The formula (2.9) in Kim--Lee~\cite{KL23} shows that
\begin{equation}\label{eq-test-E}
    |\mathcal{E}(u, \varphi_j-\varphi)| \leq C(u) \|\varphi_j-\varphi\|_{W^{s, p}(G)}.
\end{equation}
On the other hand, by using \eqref{eq-str} we have
\begin{equation*}
    \left| \int_\Omega b(x, u(x)) (\varphi_j-\varphi)(x) \,\mathrm{d}x \right|
    \leq \int_{G} (b_1|u|^{p-1} + b_2)|\varphi_j-\varphi| \,\mathrm{d}x.
\end{equation*}
We claim that
\begin{equation}\label{eq-b-Holder}
    \int_{G} (b_1|u|^{p-1} + b_2)|w| \,\mathrm{d}x \leq C \left( \|u\|_{W^{s, p}(G)}^{p-1} + 1 \right) \|w\|_{W^{s, p}(G)}
\end{equation}
for any $w \in W^{s, p}(G)$, where $C=C(n, s, p, \varepsilon, \|b_1\|_{L^{n/(sp-\varepsilon)}(G)}, \|b_2\|_{L^{n/(sp-\varepsilon)}(G)}, G)>0$. Indeed, if $sp<n$, then it follows from H\"older's inequality and the fractional Sobolev inequality that
\begin{equation*}
    \int_{G} b_1|u|^{p-1}|w| \,\mathrm{d}x \leq |G|^{\frac{\varepsilon}{n}} \|b_1\|_{L^{\frac{n}{sp-\varepsilon}}(G)} \|u\|_{L^{p^{\ast}_{s}}(G)}^{p-1} \|w\|_{L^{p^{\ast}_s}(G)} \leq C \|u\|_{W^{s, p}(G)}^{p-1} \|w\|_{W^{s, p}(G)}
\end{equation*}
and that
\begin{equation*}
    \int_G b_2|w| \,\mathrm{d}x
    \leq |G|^{1-\frac{sp-\varepsilon}{n}-\frac{1}{p^{\ast}_s}} \|b_2\|_{L^{\frac{n}{sp-\varepsilon}}(G)} \|w\|_{L^{p^{\ast}_s}(G)}
    \leq C \|w\|_{W^{s, p}(G)}.
\end{equation*}
If $sp=n$, then we let
\begin{equation}\label{eq-sp}
\tilde{s} \in \left( s\left(1-\frac{\varepsilon}{2n}\right), s \right) \quad\text{and}\quad \tilde{p}=\frac{n}{\tilde{s}+\frac{\varepsilon}{2p}}
\end{equation}
so that $\tilde{p}<p$ and $\tilde{s}\tilde{p}<n$. Then similar computations as above together with Lemma~\ref{lem-Holder} (applied with $u\coloneqq 1$ and $v\coloneqq u$) and Jensen's inequality show that
\begin{equation*}
\begin{split}
    \int_{G} b_1|u|^{p-1}|w| \,\mathrm{d}x
    &\leq |G|^{\frac{\varepsilon}{2n}} \|b_1\|_{L^{\frac{n}{sp-\varepsilon}}(G)} \|u\|_{L^{\tilde{p}^{\ast}_{\tilde{s}}}(G)}^{p-1} \|w\|_{L^{\tilde{p}^{\ast}_{\tilde{s}}}(G)} \\
    &\leq C \|u\|_{W^{\tilde{s}, \tilde{p}}(G)}^{p-1} \|w\|_{W^{\tilde{s}, \tilde{p}}(G)} \leq C \|u\|_{W^{s, p}(G)}^{p-1} \|w\|_{W^{s, p}(G)}
\end{split}
\end{equation*}
and that
\begin{equation*}
    \int_G b_2|w| \,\mathrm{d}x
    \leq |G|^{\frac{\varepsilon}{n}-\frac{1}{\tilde{p}^{\ast}_{\tilde{s}}}} \|b_2\|_{L^{\frac{n}{sp-\varepsilon}}(G)} \|w\|_{L^{\tilde{p}^{\ast}_{\tilde{s}}}(G)} \leq C \|w\|_{W^{s, p}(G)}.
\end{equation*}
In either case, the claim \eqref{eq-b-Holder} holds.

Therefore, by using \eqref{eq-test-E} and \eqref{eq-b-Holder} applied with $w=\varphi_j-\varphi$, we conclude that
\begin{equation*}
\mathcal{E}(u, \varphi_j) + \int_\Omega b(x, u(x))\varphi_j(x) \,\mathrm{d}x \to \mathcal{E}(u, \varphi) + \int_\Omega b(x, u(x))\varphi(x) \,\mathrm{d}x
\end{equation*}
as $j \to \infty$.
\end{proof}

As a consequence of Proposition~\ref{prop-test}, we obtain the following result.

\begin{corollary}\label{cor-sol}
A function $u$ is a weak solution of \eqref{eq-main} in $\Omega$ if and only if $u$ is both a weak subsolution and a weak supersolution of \eqref{eq-main} in $\Omega$.
\end{corollary}

\begin{proof}
    Let $\varphi \in C^\infty_c(\Omega)$, then $\varphi=\varphi_+-\varphi_-$, where $\varphi_+\coloneqq\max\{\varphi, 0\}$ and $\varphi_-\coloneqq \max\{-\varphi, 0\}$. Since $\varphi_\pm$ are nonnegative functions such that $\varphi_\pm \in W^{s, p}_{\mathrm{loc}}(\Omega)$ and $\supp{\varphi_{\pm}} \Subset \Omega$, the desired result follows from Proposition~\ref{prop-test}.
\end{proof}

Proposition~\ref{prop-test} is useful especially when we test the equation \eqref{eq-main} by functions of the form $u\eta$, where $\eta$ is a cut-off function. The following lemma shows that such a function is a legal test function.

\begin{lemma}\label{lem-test-prod}
Let $u \in W^{s, p}_{\mathrm{loc}}(\Omega)$ and $\eta \in C^\infty_c(\Omega)$, then $u\eta \in W^{s, p}(\Omega)$ and $\supp{(u\eta)} \Subset \Omega$.
\end{lemma}

\begin{proof}
Let $S=\supp{(u\eta)}$, then it is clear that $S \subset \supp{\eta} \Subset \Omega$ and that
\begin{equation*}
\|u\eta\|_{L^p(\Omega)} \leq \|\eta\|_{L^\infty(S)} \|u\|_{L^p(S)} < \infty.
\end{equation*}
In order to prove that $[u\eta]_{W^{s, p}(\Omega)} < \infty$, we fix an open set $G$ such that $S \Subset G \Subset \Omega$. Then
\begin{equation*}
[u\eta]_{W^{s, p}(\Omega)}^p = \int_G \int_G \frac{|(u\eta)(x)-(u\eta)(y)|^p}{|x-y|^{n+sp}} \,\mathrm{d}y \,\mathrm{d}x + 2 \int_G \int_{\Omega \setminus G} \frac{|(u\eta)(x)|^p}{|x-y|^{n+sp}} \,\mathrm{d}y \,\mathrm{d}x \eqqcolon I_1 + I_2.
\end{equation*}
We see that
\begin{align*}
I_1
&\leq C \int_G \int_G \frac{|u(x)-u(y)|^p}{|x-y|^{n+sp}} |\eta(y)|^p \,\mathrm{d}y \,\mathrm{d}x + C \int_G |u(x)|^p \int_G \frac{|\eta(x)-\eta(y)|^p}{|x-y|^{n+sp}} \,\mathrm{d}y \,\mathrm{d}x \\
&\leq C \|\eta\|_{L^\infty(G)}^p [u]_{W^{s, p}(G)}^p + C(\mathrm{diam}\,G)^{p(1-s)} \|\nabla \eta\|_{L^\infty(G)}^p \|u\|_{L^p(G)}^p
\end{align*}
and that
\begin{equation*}
I_2 \leq 2\|\eta\|_{L^\infty(S)}^p \int_S |u(x)|^p \int_{\mathbb{R}^n \setminus B_d(x)} \frac{\mathrm{d}y}{|x-y|^{n+sp}} \,\mathrm{d}x \leq \frac{C}{d^{sp}} \|\eta\|_{L^\infty(S)}^p \|u\|_{L^p(S)}^p
\end{equation*}
for some $C=C(n, s, p) > 0$, where $d=\mathrm{dist}(S, G^c)$. Since $u \in W^{s ,p}_{\mathrm{loc}}(\Omega)$, we conclude that $u\eta \in W^{s, p}(\Omega)$.
\end{proof}

Since we are interested in weak solutions with an isolated singularity, we need to deal with functions in $W^{s, p}_{\mathrm{loc}}(\Omega \setminus \{x_0\})$ for some $x_0 \in \Omega$. The following lemma shows that one can regularize such functions by truncating them when they have large values near the singularity.

\begin{lemma}\label{lem-test-trunc}
Let $x_0 \in \Omega$ and let $u \in W^{s, p}_{\mathrm{loc}}(\Omega \setminus \{x_0\})$ be such that $u \geq 1$ in a neighborhood of $x_0$. Then $\min\{u, 1\} \in W^{s, p}_{\mathrm{loc}}(\Omega)$.
\end{lemma}

\begin{proof}
It is enough to prove that $\varphi\coloneqq\min\{u, 1\} \in W^{s, p}(G)$ for each open set $G \Subset \Omega$ containing $x_0$. Take $r>0$ sufficiently small so that $B_r=B_r(x_0) \Subset G$ and $u \geq 1$ on $B_r$. Then clearly $\|\varphi\|_{L^p(G)}^p \leq \|u\|_{L^p(G\setminus B_r)}^p + |B_r|<\infty$. Moreover, it follows from $|\varphi(x)-\varphi(y)|\leq |u(x)-u(y)|$ that
\begin{equation*}
\begin{split}
[\varphi]_{W^{s, p}(G)}^p
&\leq [\varphi]_{W^{s, p}(G \setminus B_{r/2})}^p + 2\int_{G \setminus B_{r/2}} \int_{B_{r/2}} \frac{|\varphi(x)-\varphi(y)|^p}{|x-y|^{n+sp}} \,\mathrm{d}y\,\mathrm{d}x \\
&\leq [u]_{W^{s, p}(G \setminus B_{r/2})}^p + 2\int_{G \setminus B_r} \int_{B_{r/2}} \frac{|\varphi(x)-1|^p}{|x-y|^{n+sp}} \,\mathrm{d}y\,\mathrm{d}x \\
&\leq [u]_{W^{s, p}(G \setminus B_{r/2})}^p + Cr^{-sp} \int_{G \setminus B_r} (|u(x)|^p+1) \,\mathrm{d}x < \infty,
\end{split}
\end{equation*}
which proves that $\varphi \in W^{s, p}(G)$.
\end{proof}

Instead of regularizing functions in $W^{s, p}_{\mathrm{loc}}(\Omega \setminus \{x_0\})$, one can use a function that is constant near the singularity.

\begin{lemma}\label{lem-test-const}
Let $x_0 \in \Omega$ and $u \in W^{s, p}_{\mathrm{loc}}(\Omega \setminus \{x_0\}) \cap L^{p-1}_{sp}(\mathbb{R}^n)$. If $\varphi \in W^{s, p}_{\mathrm{loc}}(\Omega \setminus \{x_0\}) \cap L^\infty(\Omega)$ is such that $\mathrm{supp}\,{\varphi} \Subset \Omega$ and $\varphi=1$ in a neighborhood of $x_0$, then $\mathcal{E}(u, \varphi)$ is finite.
\end{lemma}

\begin{proof}
Let $S=\supp\varphi$ and take $r>0$ sufficiently small so that $\varphi=1$ on $B_r=B_r(x_0)$ and that $B_r \Subset \mathrm{int}\,S$. We fix an open set $G$ such that $S \Subset G \Subset \Omega$. Then
\begin{align*}
|\mathcal{E}(u, \varphi)|
&\leq \Lambda \iint_{(G^c \times G^c)^c \setminus (B_{r/2} \times B_{r/2})} \frac{|u(x)-u(y)|^{p-1}|\varphi(x)-\varphi(y)|}{|x-y|^{n+sp}} \,\mathrm{d}y \,\mathrm{d}x \\
&= \Lambda \int_{G \setminus B_{r/2}} \int_{G \setminus B_{r/2}} \frac{|u(x)-u(y)|^{p-1}|\varphi(x)-\varphi(y)|}{|x-y|^{n+sp}} \,\mathrm{d}y \,\mathrm{d}x \\
&\quad + 2\Lambda \int_{B_{r/2}} \int_{\mathbb{R}^n \setminus B_{r/2}} \frac{|u(x)-u(y)|^{p-1}|1-\varphi(y)|}{|x-y|^{n+sp}} \,\mathrm{d}y \,\mathrm{d}x \\
&\quad + 2\Lambda \int_{G \setminus B_{r/2}} \int_{G^c} \frac{|u(x)-u(y)|^{p-1}|\varphi(x)|}{|x-y|^{n+sp}} \,\mathrm{d}y \,\mathrm{d}x \eqqcolon I_1 + I_2 + I_3,
\end{align*}
where $\Lambda$ is the ellipticity constant given in \eqref{eq-ellipticity}.

Since $G \setminus B_{r/2} \Subset \Omega \setminus \{x_0\}$, it follows that $u \in W^{s, p}(G\setminus B_{r/2})$. Thus, H\"older's inequality shows that
\begin{equation*}
I_1 \leq \Lambda [u]_{W^{s, p}(G \setminus B_{r/2})}^{p-1} [\varphi]_{W^{s, p}(G \setminus B_{r/2})} < \infty.
\end{equation*}
For $I_2$ and $I_3$, we use Lemma~\ref{lem-tail} to see that
\begin{align*}
I_2
&\leq C \|1-\varphi\|_{L^\infty(\mathbb{R}^n)} \int_{B_{r/2}} \int_{\mathbb{R}^n \setminus B_r} \frac{|u(x)|^{p-1} + |u(y)|^{p-1}}{|x-y|^{n+sp}} \,\mathrm{d}y \,\mathrm{d}x \\
&\leq \frac{C}{r^{sp}} \left( 1+\|\varphi\|_{L^\infty(S)} \right) \left( \|u\|_{L^{p-1}(B_{r/2})}^{p-1} + r^n \mathrm{Tail}^{p-1}(u; x_0, r/2) \right)
\end{align*}
and that
\begin{align*}
I_3
&\leq C \|\varphi\|_{L^\infty(S)} \int_S \int_{\mathbb{R}^n \setminus B_d(x_0)} \frac{|u(x)|^{p-1} + |u(y)|^{p-1}}{|x-y|^{n+sp}} \,\mathrm{d}y \,\mathrm{d}x \\
&\leq \frac{C}{d^{sp}} \|\varphi\|_{L^\infty(S)} \left( \|u\|_{L^{p-1}(S)}^{p-1} + |S| \left( 1+\frac{\mathrm{diam}\,S}{d} \right)^{n+sp} \mathrm{Tail}^{p-1}(u; x_0, d) \right),
\end{align*}
where $d=\mathrm{dist}(S, G^c)$. Since $u \in L^{p-1}_{sp}(\mathbb{R}^n)$, both $I_2$ and $I_3$ are finite.
\end{proof}

\subsection{\texorpdfstring{$(s, p)$}{(s, p)}-capacity}\label{sec-capacity}

In this section, we recall the definition of the $(s, p)$-capacity and collect some basic properties of it. Let us begin with the definition.

\begin{definition}
	The \emph{$(s, p)$-capacity} of a compact set $K \subset \Omega$ is defined by
	\begin{equation*}
		\mathrm{cap}_{s, p}(K, \Omega) = \inf_{u \in W(K, \Omega)} {[u]_{W^{s, p}(\mathbb{R}^n)}^p},
	\end{equation*}
	where $W(K, \Omega) = \{u \in C^\infty_c(\Omega): u \geq 1~\text{on}~K \}$. For open sets $G \subset \Omega$
 \begin{equation}   \label{eq-cap-G}
		\mathrm{cap}_{s, p}(G,\Omega) \coloneqq \sup_{\substack{K \text{ compact}\\ K \subset G}} \mathrm{cap}_{s, p}(K,\Omega),
\end{equation}
and for arbitrary sets $E \subset \Omega$,
\begin{equation}  \label{eq-cap-E}
		\mathrm{cap}_{s, p}(E,\Omega) \coloneqq \inf_{\substack{G \text{ open}\\ E \subset G \subset \Omega}} \mathrm{cap}_{s, p}(G,\Omega).
\end{equation}
\end{definition}

As pointed out in Bj\"orn--Bj\"orn--Kim~\cite[Section~5]{BBK24}, \eqref{eq-cap-G} and \eqref{eq-cap-E} do not change the definition of capacities for compact sets.

The following estimates for the $(s, p)$-capacity of balls and annulus will be useful in the sequel.

\begin{lemma}\label{lem-cap-ball}
	Let $0<r \leq R/2$. Then
	\begin{equation*}
		\mathrm{cap}_{s, p}(\overline{B}_r \setminus B_{r/2}, B_R) \eqsim \mathrm{cap}_{s,p}(\overline{B}_r, B_R) \eqsim 
		\begin{cases}
			\left(r^{-\frac{n-sp}{p-1}}-R^{-\frac{n-sp}{p-1}} \right)^{1-p} &\text{if } sp<n, \\
			(\log(R/r))^{1-p}  &\text{if } sp=n,
		\end{cases}
	\end{equation*}
	where the comparable constants depend only on $n$, $s$ and $p$.
\end{lemma}

\begin{proof}
The second comparability is contained in Kim--Lee--Lee~\cite[Lemma~2.17]{KLL23}, so it is enough to prove the first comparability. Since the capacity is monotone increasing with respect to the argument, it suffices to prove that
	\begin{equation*}
		\mathrm{cap}_{s, p}(\overline{B}_r \setminus B_{r/2}, B_R) \gtrsim \mathrm{cap}_{s, p}(\overline{B}_r, B_R).
	\end{equation*}
	Indeed, since there exists a ball $B=B_{r/4}(y_0)$ contained in the annulus $\overline{B}_r \setminus B_{r/2}$, we have
	\begin{equation*}
		\mathrm{cap}_{s, p}(\overline{B}_r \setminus B_{r/2}, B_R) \geq \mathrm{cap}_{s, p}(B, B_R) \gtrsim \mathrm{cap}_{s, p}(\overline{B}_r, B_R).
	\end{equation*}
\end{proof}

A set of $(s, p)$-capacity zero plays a key role in the main theorems. Let us provide a precise definition of it.

\begin{definition}
    A set $E \subset \mathbb{R}^n$ is of \emph{$(s, p)$-capacity zero} if $\mathrm{cap}_{s, p}(E \cap \Omega, \Omega)=0$ for all open $\Omega \subset \mathbb{R}^n$.
\end{definition}

A set of $(s, p)$-capacity zero in particular has measure zero.

\begin{lemma}\label{lem-measurezero}
    If $E \subset \mathbb{R}^n$ is of $(s, p)$-capacity zero, then $|E|=0$.
\end{lemma}

\begin{proof}
Let $\Omega$ be a bounded open set and let $\varepsilon >0$. Since $\mathrm{cap}_{s, p}(E \cap \Omega, \Omega)=0$, there exists an open neighborhood $G \subset \Omega$ of $E \cap \Omega$ such that $\mathrm{cap}_{s, p}(G, \Omega) < \varepsilon$. Let $K \subset G$ be a compact set and take $\varphi \in W(K, \Omega)$ such that $[\varphi]_{W^{s, p}(\mathbb{R}^n)}^p \leq \mathrm{cap}_{s, p}(K, \Omega) + \varepsilon < 2\varepsilon$. Thus, the fractional Poincar\'e inequality (see e.g.\ Bj\"orn--Bj\"orn--Kim~\cite[(5.6)]{BBK24}) shows that
\begin{equation*}
    |K| \leq \int_\Omega |\varphi|^p \,\mathrm{d}x \leq C [\varphi]_{W^{s, p}(\mathbb{R}^n)} < C\varepsilon,
\end{equation*}
which in turn implies that $|G| \leq C\varepsilon$. Therefore, we conclude that $|E|=0$.
\end{proof}

In case $E$ is a singleton, we have the following result.

\begin{lemma}\label{lem-cap-point}
	Let $x_0 \in \mathbb{R}^n$. Then $\{x_0\}$ is of $(s, p)$-capacity zero if and only if $sp \leq n$.
\end{lemma}

\begin{proof}
It follows from \cite[Lemma~5.5]{BBK24} that $sp\leq n$ if and only if $\mathrm{cap}_{s, p}(\{x_0\}, \Omega)=0$ for every bounded open $\Omega$ containing $x_0$. Since $\mathrm{cap}_{s, p}(\{x_0\}, \cdot)$ is decreasing with respect to the set inclusion, the desired result follows.
\end{proof}

\section{Regularity estimates}\label{sec-regularity}

In this section, we provide several regularity estimates such as Caccioppoli-type estimates, H\"older estimates and the Harnack inequality, which play a fundamental role in the analysis of singularities for weak solutions of \eqref{eq-main}.

\subsection{Caccioppoli-type estimates}

Various Caccioppoli estimates are known in the literature, but we need special types of estimates due to the lack of regularity of weak solutions. Such estimates will be used in a crucial way in the sequel. For instance, if $u$ is a weak solution of \eqref{eq-main} in $\Omega \setminus E$ for some measure zero set $E$, one can regularize $u$ under some appropriate assumptions by using Caccioppoli-type estimates so that $u$ is a weak solution to the same equation in $\Omega$ or has nice regularity properties in $\Omega$.

In order to state the Caccioppoli-type estimates, we fix $\beta \geq \beta_0>0$, $\gamma = \beta+p-1$ and $\gamma_0 = \beta_0+p-1$. For $l>d>0$, we define the functions $F, G: [d, \infty) \to [0, \infty)$ by
\begin{equation}\label{eq-F}
F(t)=
\begin{cases}
t^{\gamma/p} &\text{if}~d\leq t \leq l, \\
l^{\gamma/p} (\frac{\gamma}{\gamma_0}(\frac{t}{l})^{\gamma_0/p} + \frac{\gamma_0-\gamma}{\gamma_0}) &\text{if}~t> l,
\end{cases}
\end{equation}
and
\begin{equation*}
G(t)=F(t)F'(t)^{p-1} - F(d)F'(d)^{p-1}.
\end{equation*}
Note that $F$ is the minimum of two polynomials.

\begin{lemma}\label{lem-Caccio1}
Let $E$ be a relatively closed set in $\Omega$ with measure zero and let $\bar{\eta}$ be a nonnegative smooth function vanishing in a neighborhood of $E$. Let $B_r(x_0) \subset B_R(x_0)\subset \Omega$ and let $\eta \in C^\infty_c(B_r(x_0))$ be a nonnegative function. Assume that
\begin{equation}\label{eq-d}
d > \left( R^{\varepsilon}\|b_2\|_{L^{n/(sp-\varepsilon)}(B_R(x_0))} \right)^{1/(p-1)}.
\end{equation}
If $u$ is a weak subsolution of \eqref{eq-main} in $B_R(x_0) \setminus E$, then
\begin{equation}\label{eq-Caccio1}
\begin{split}
&\int_{B_r(x_0)} \int_{B_r(x_0)} \frac{|F(\bar{u}(x))-F(\bar{u}(y))|^p}{|x-y|^{n+sp}} \max\{ (\eta\bar{\eta})(x), (\eta\bar{\eta})(y) \}^p \,\mathrm{d}y \,\mathrm{d}x \\
&\leq C \left( 1+\gamma^{p-1}+\left( \frac{\gamma}{\beta_0} \right)^p \right) \int_{B_r(x_0)} F(\bar{u}(x))^p \int_{B_r(x_0)} \frac{|(\eta\bar{\eta})(x)-(\eta\bar{\eta})(y)|^p}{|x-y|^{n+sp}} \,\mathrm{d}y\,\mathrm{d}x \\
&\quad + C \left( \gamma^{p-1} + \gamma^{2sp(p-1)/\varepsilon} \right) r^{-sp} \|F(\bar{u})\eta\bar{\eta}\|_{L^p(B_r(x_0))}^p \\
&\quad + C \frac{\gamma}{\beta_0} \left( \sup_{x \in \mathrm{supp} \, \eta} \int_{\mathbb{R}^n \setminus B_r(x_0)} \frac{\bar{u}^{p-1}(y)}{|x-y|^{n+sp}} \,\mathrm{d}y \right) \int_{B_r(x_0)} F(\bar{u}) F'(\bar{u})^{p-1} (\eta\bar{\eta})^p \,\mathrm{d}x,
\end{split}
\end{equation}
where $\bar{u}=u_++d$ and $C=C(n, s, p, \Lambda, \varepsilon, R^{\varepsilon}\|b_1\|_{L^{n/(sp-\varepsilon)}(B_R(x_0))})>0$.
\end{lemma}

Before we provide the proof of Lemma~\ref{lem-Caccio1}, we prove the following algebraic inequality as a discrete version of
\begin{equation*}
|\nabla u|^{p-2} \nabla u \cdot \nabla \left( \frac{\gamma}{\beta_0}G(\bar{u}) \eta^p \right) \geq c \, |\nabla (F(\bar{u}) \eta)|^p - C \left( 1+\left( \frac{\gamma}{\beta_0} \right)^p \right) |F(\bar{u})|^p |\nabla \eta|^p,
\end{equation*}
where the constants $c, C>0$ depend only on $p$.

\begin{lemma}\label{lem-alg-ineq1}
Let $a, b \in \mathbb{R}$, $A=a_++d, B=b_++d$ and $\eta_1, \eta_2 \geq 0$. Then
\begin{equation*}
\begin{split}
&\frac{\gamma}{\beta_0} |a-b|^{p-2}(a-b) (G(A) \eta_1^p - G(B) \eta_2^p) \\
&\geq c \left| F(A) - F(B) \right|^p \max\{\eta_1, \eta_2\}^p - C \left( 1+ \left( \frac{\gamma}{\beta_0} \right)^p \right) \max \lbrace F(A), F(B) \rbrace^p |\eta_1-\eta_2|^p,
\end{split}
\end{equation*}
where $c$ and $C$ are positive constants depending only on $p$.
\end{lemma}

\begin{proof}
We may assume that $a>b$. Note that since $a-b \geq A-B$ and $(a-b)^{p-1}G(B)=(A-B)^{p-1}G(B)$, we have
\begin{align*}
L
&\coloneqq \frac{\gamma}{\beta_0} (a-b)^{p-1} (G(A) \eta_1^p - G(B) \eta_2^p) \\
&= \frac{\gamma}{\beta_0} (a-b)^{p-1} (G(A) - G(B))\eta_1^p + \frac{\gamma}{\beta_0} (a-b)^{p-1} G(B) (\eta_1^p - \eta_2^p) \\
&\geq \frac{\gamma}{\beta_0} (A-B)^{p-1} (G(A) - G(B))\eta_1^p + \frac{\gamma}{\beta_0} (A-B)^{p-1} G(B) (\eta_1^p - \eta_2^p) \\
&= \frac{\gamma}{\beta_0} (A-B)^{p-1} (G(A) - G(B))\eta_2^p + \frac{\gamma}{\beta_0} (A-B)^{p-1} G(A) (\eta_1^p - \eta_2^p).
\end{align*}
A straightforward computation shows that
\begin{alignat*}{2}
G'(t) &= p\beta \gamma^{-1} (F')^p &&\text{if}~t \leq l, \\
G'(t) &\geq p\beta_0 \gamma^{-1} (F')^p \quad&&\text{if}~t \geq l.
\end{alignat*}
It thus follows from Jensen's inequality that
\begin{equation*}
(F(A)-F(B))^p \leq (A-B)^p \fint_B^A (F')^p(t) \,\mathrm{d}t \leq \frac{\gamma}{p\beta_0} (A-B)^{p-1}(G(A)-G(B)).
\end{equation*}
This together with the inequalities $G \leq F(F')^{p-1}$ and $|\eta_1^p-\eta_2^p| \leq p|\eta_1-\eta_2| \max\{\eta_1, \eta_2\}^{p-1}$ shows that
\begin{equation*}
	\begin{aligned}
		L &\geq p(F(A)-F(B))^p \eta_2^p - \frac{p\gamma}{\beta_0} F(A)|\eta_1 - \eta_2| \left( (A-B) F'(A) \max\{\eta_1, \eta_2\} \right)^{p-1} \quad\text{and}  \\
		L &\geq p(F(A)-F(B))^p \eta_1^p - \frac{p\gamma}{\beta_0} F(B)|\eta_1 - \eta_2| \left( (A-B) F'(B) \max\{\eta_1, \eta_2\} \right)^{p-1}.
	\end{aligned}
\end{equation*}
We observe that the minimum of $F'$ on $[B, A]$ is either $F'(A)$ or $F'(B)$, so
\begin{equation*}
F(A)-F(B) \geq (A-B) \min\{ F'(A), F'(B) \}.
\end{equation*}
An application of Young's inequality shows that
\begin{align*}
L
&\geq p(F(A)-F(B))^p \min\{\eta_1, \eta_2\}^p - \frac{p}{2^{2p-1}} (F(A)-F(B))^p \max\{\eta_1, \eta_2\}^p \\
&\quad - C\left( \frac{\gamma}{\beta_0} \right)^p \max\{F(A), F(B)\}^p |\eta_1-\eta_2|^p,
\end{align*}
where $C=C(p)>0$. The desired result now follows by applying Lemma~A.3 in Kim--Lee--Lee~\cite{KLL23}.
\end{proof}

\begin{proof}[Proof of Lemma~\ref{lem-Caccio1}]
Let us write $B_r=B_r(x_0)$ for simplicity. We define $\varphi=G(\bar{u}) (\eta\bar{\eta})^p \geq 0$. 
By Proposition~\ref{prop-test}, $\varphi$ can be used as a test function for \eqref{eq-main} in $B_R \setminus E$. Thus, we have
\begin{equation*}
\mathcal{E}(u, \varphi) + \int_{B_R \setminus E} b(x, u(x)) \varphi(x) \,\mathrm{d}x \leq 0.
\end{equation*}
Let $L$ denote the left-hand side of \eqref{eq-Caccio1}. An application of Lemma~\ref{lem-alg-ineq1} with $a=\bar{u}(x)$, $b=\bar{u}(y)$, $\eta_1=(\eta\bar{\eta})(x)$ and $\eta_2=(\eta\bar{\eta})(y)$ shows that
\begin{equation*}
\begin{split}
L
&\leq C_1 \left( 1+\left( \frac{\gamma}{\beta_0} \right)^p \right) \int_{B_r} F(\bar{u}(x))^p \int_{B_r} \frac{|(\eta\bar{\eta})(x)-(\eta\bar{\eta})(y)|^p}{|x-y|^{n+sp}} \,\mathrm{d}y\,\mathrm{d}x \\
&\quad - C_2 \frac{\gamma}{\beta_0} \int_{B_r} \int_{\mathbb{R}^n \setminus B_r} |\bar{u}(x)-\bar{u}(y)|^{p-2} (\bar{u}(x)-\bar{u}(y)) \varphi(x) k(x, y) \,\mathrm{d}y \,\mathrm{d}x \\
&\quad + \int_{B_r} (b_1u_+^{p-1} + b_2) G(\bar{u})(\eta \bar{\eta})^p \,\mathrm{d}x,
\end{split}
\end{equation*}
where $C_1$ and $C_2$ are positive constants depending only on $p$ and $\Lambda$.

Since $G\leq F(F')^{p-1}$ and $tF'(t) \leq \frac{\gamma}{p}F(t)$, we obtain that
\begin{align*}
&-\int_{B_r} \int_{\mathbb{R}^n \setminus B_r} |\bar{u}(x)-\bar{u}(y)|^{p-2} (\bar{u}(x)-\bar{u}(y)) \varphi(x) k(x, y) \,\mathrm{d}y \,\mathrm{d}x \\
&\leq \int_{B_r} \int_{\mathbb{R}^n \setminus B_r} (\bar{u}(y)-\bar{u}(x))^{p-1} G(\bar{u}(x)) (\eta\bar{\eta})^p(x) {\bf 1}_{\lbrace \bar{u}(x) \leq \bar{u}(y) \rbrace} k(x, y) \,\mathrm{d}y \,\mathrm{d}x \\
&\leq \Lambda \int_{B_r} F(\bar{u}(x)) F'(\bar{u}(x))^{p-1} (\eta\bar{\eta})^p(x) \int_{\mathbb{R}^n \setminus B_r} \frac{\bar{u}^{p-1}(y)}{|x-y|^{n+sp}} \,\mathrm{d}y \,\mathrm{d}x \\
&\leq \Lambda \left( \sup_{x \in \mathrm{supp} \, \eta} \int_{\mathbb{R}^n \setminus B_r} \frac{\bar{u}^{p-1}(y)}{|x-y|^{n+sp}} \,\mathrm{d}y \right) \int_{B_r} F(\bar{u}) F'(\bar{u})^{p-1} (\eta\bar{\eta})^p \,\mathrm{d}x
\end{align*}
and that
\begin{equation*}
\int_{B_r} (b_1u_+^{p-1} + b_2) G(\bar{u})(\eta \bar{\eta})^p \,\mathrm{d}x \leq \left( \frac{\gamma}{p} \right)^{p-1} \int_{B_r} (b_1+d^{1-p}b_2) (F(\bar{u}) \eta\bar{\eta})^p \,\mathrm{d}x.
\end{equation*}

In the rest of the proof, we focus on the estimate of the last integral in the display above. Let us first consider the case $sp<n$. Since \eqref{eq-d} implies
\begin{equation}\label{eq-b}
R^\varepsilon \|b_1+d^{1-p}b_2\|_{L^{\frac{n}{sp-\varepsilon}}(B_R)} < R^\varepsilon \|b_1\|_{L^{\frac{n}{sp-\varepsilon}}(B_R)} + 1 \leq C,
\end{equation}
H\"older's inequality shows that
\begin{equation*}
\begin{split}
\int_{B_r} (b_1+d^{1-p}b_2) (F(\bar{u}) \eta\bar{\eta})^p \,\mathrm{d}x
&\leq \|b_1+d^{1-p}b_2\|_{L^{\frac{n}{sp-\varepsilon}}(B_R)} \|F(\bar{u})\eta\bar{\eta}\|_{L^p(B_r)}^{\varepsilon/s} \|F(\bar{u})\eta\bar{\eta}\|_{L^{p^{\ast}_{s}}(B_r)}^{p-\varepsilon/s} \\
&\leq C r^{-\varepsilon} \|F(\bar{u})\eta\bar{\eta}\|_{L^p(B_r)}^{\varepsilon/s} \|F(\bar{u})\eta\bar{\eta}\|_{L^{p^{\ast}_{s}}(B_r)}^{p-\varepsilon/s}.
\end{split}
\end{equation*}
Moreover, it follows from Theorem~\ref{thm-Sobolev} and
\begin{equation}\label{eq-triangle}
\begin{split}
|(F(\bar{u})\eta\bar{\eta})(x)-(F(\bar{u})\eta\bar{\eta})(y)|
&\leq |F(\bar{u}(x))-F(\bar{u}(y))| \max\{(\eta\bar{\eta})(x), (\eta\bar{\eta})(y)\} \\
&\quad + \max\{F(\bar{u}(x)), F(\bar{u}(y)) \} |(\eta\bar{\eta})(x)-(\eta\bar{\eta})(y)|
\end{split}
\end{equation}
that
\begin{equation*}
\begin{split}
\|F(\bar{u})\eta\bar{\eta}\|_{L^{p^{\ast}_{s}}(B_r)} \leq C \left( [F(\bar{u})\eta\bar{\eta}]_{W^{s, p}(B_r)} + r^{-s}\|F(\bar{u})\eta\bar{\eta}\|_{L^p(B_r)} \right) \leq C(L + I_1 + I_2)^{1/p},
\end{split}
\end{equation*}
where
\begin{equation*}
I_1 \coloneqq \int_{B_r} F(\bar{u}(x))^p \int_{B_r} \frac{|(\eta\bar{\eta})(x)-(\eta\bar{\eta})(y)|^p}{|x-y|^{n+sp}} \,\mathrm{d}y\,\mathrm{d}x \quad\text{and}\quad I_2 \coloneqq r^{-sp} \|F(\bar{u})\eta\bar{\eta}\|_{L^p(B_r)}^p.
\end{equation*}

Combining all the estimates above and using Young's inequality yields
\begin{equation}\label{eq-Young}
\begin{split}
L
&\leq C \left( 1+\left( \frac{\gamma}{\beta_0} \right)^p \right) I_1 + C\gamma^{p-1} I_2^{\frac{\varepsilon}{sp}} (L+I_1+I_2)^{1-\frac{\varepsilon}{sp}} \\
&\quad + C \frac{\gamma}{\beta_0} \left( \sup_{x \in \mathrm{supp} \, \eta} \int_{\mathbb{R}^n \setminus B_r} \frac{\bar{u}^{p-1}(y)}{|x-y|^{n+sp}} \,\mathrm{d}y \right) \int_{B_r} F(\bar{u}) F'(\bar{u})^{p-1} (\eta\bar{\eta})^p \,\mathrm{d}x \\
&\leq C\gamma^{p-1} I_2^{\frac{\varepsilon}{sp}} L^{1-\frac{\varepsilon}{sp}} + C \left( 1+\gamma^{p-1}+\left( \frac{\gamma}{\beta_0} \right)^p \right) I_1 + C\gamma^{p-1} I_2 \\
&\quad + C \frac{\gamma}{\beta_0} \left( \sup_{x \in \mathrm{supp} \, \eta} \int_{\mathbb{R}^n \setminus B_r} \frac{\bar{u}^{p-1}(y)}{|x-y|^{n+sp}} \,\mathrm{d}y \right) \int_{B_r} F(\bar{u}) F'(\bar{u})^{p-1} (\eta\bar{\eta})^p \,\mathrm{d}x,
\end{split}
\end{equation}
where $C=C(n,s,p,\Lambda, R^{\varepsilon} \|b_1\|_{L^{n/(sp-\varepsilon)}(B_R)})>0$. The desired estimate \eqref{eq-Caccio1} now follows from Lemma~2 in Serrin~\cite{Ser64} and a simple observation that $\gamma^{sp(p-1)/\varepsilon} \leq \gamma^{p-1}+\gamma^{2sp(p-1)/\varepsilon}$.

Let us next consider the case $sp=n$. We set $\tilde{s}$ and $\tilde{p}$ as in \eqref{eq-sp} so that $\tilde{p}<p$ and $\tilde{s}\tilde{p}<n$. Then it follows from H\"older's inequality and \eqref{eq-b} that
\begin{equation*}
\begin{split}
\int_{B_r} (b_1+d^{1-p}b_2) (F(\bar{u})\eta\bar{\eta})^p \,\mathrm{d}x
&\leq \|b_1+d^{1-p}b_2\|_{L^{\frac{n}{n-\varepsilon}}(B_R)} \|F(\bar{u})\eta\bar{\eta}\|_{L^{\tilde{p}}(B_r)}^{\varepsilon/(2\tilde{s})} \|F(\bar{u})\eta\bar{\eta}\|_{L^{\tilde{p}^{\ast}_{\tilde{s}}}(B_r)}^{p-\varepsilon/(2\tilde{s})} \\
&\leq C r^{-\varepsilon} \left( r^{n\frac{p-\tilde{p}}{p\tilde{p}}} \|F(\bar{u})\eta\bar{\eta}\|_{L^{p}(B_r)} \right)^{\varepsilon/(2\tilde{s})} \|F(\bar{u})\eta\bar{\eta}\|_{L^{\tilde{p}^{\ast}_{\tilde{s}}}(B_r)}^{p-\varepsilon/(2\tilde{s})} \\
&= C I_2^{\varepsilon/(2\tilde{s}p)} \left( r^{-\varepsilon/2} \|F(\bar{u})\eta\bar{\eta}\|_{L^{\tilde{p}^{\ast}_{\tilde{s}}}(B_r)}^p \right)^{1-\varepsilon/(2\tilde{s}p)}.
\end{split}
\end{equation*}
By applying Theorem~\ref{thm-Sobolev} and Lemma~\ref{lem-Jensen}, and then using \eqref{eq-triangle}, we obtain that
\begin{equation*}
\begin{split}
r^{-\varepsilon/2} \|F(\bar{u})\eta\bar{\eta}\|_{L^{\tilde{p}^{\ast}_{\tilde{s}}}(B_r)}^p
&\leq C r^{-\varepsilon/2} \left( [F(\bar{u})\eta\bar{\eta}]_{W^{\tilde{s}, \tilde{p}}(B_r)} + r^{-\tilde{s}}\|F(\bar{u})\eta\bar{\eta}\|_{L^{\tilde{p}}(B_r)} \right)^p \\
&\leq C r^{-\varepsilon/2 + n\frac{p-\tilde{p}}{\tilde{p}} + (s-\tilde{s})p} \left( [F(\bar{u})\eta\bar{\eta}]_{W^{s, p}(B_r)} + r^{-s} \|F(\bar{u})\eta\bar{\eta}\|_{L^p(B_r)} \right)^p \\
&\leq C (L + I_1 + I_2).
\end{split}
\end{equation*}
Combining the estimates above yields \eqref{eq-Young} with the exponent $\varepsilon/(sp)$ replaced by $\varepsilon/(2\tilde{s}p)$. The desired estimate \eqref{eq-Caccio1} then follows from \cite[Lemma~2]{Ser64} and that $\gamma^{2\tilde{s}p(p-1)/\varepsilon} \leq \gamma^{p-1}+\gamma^{2sp(p-1)/\varepsilon}$ as in the previous case.
\end{proof}

A particular choice of $\eta$ in Lemma~\ref{lem-Caccio1} yields the following lemma.

\begin{lemma}\label{lem-Caccio2}
Let $E \subset \mathbb{R}^n$ be a set of measure zero and let $\bar{\eta}$ be a nonnegative smooth function such that $\bar{\eta}$ vanishes in a neighborhood of $E$ and $0\leq\bar{\eta}\leq1$. Let $B_\rho(x_0) \Subset B_r(x_0) \subset B_R(x_0) \subset \Omega$. Assume \eqref{eq-d}. If $u$ is a weak subsolution of \eqref{eq-main} in $B_R(x_0) \setminus E$, then
\begin{equation}\label{eq-Caccio2}
\begin{split}
&\int_{B_\rho(x_0)} \int_{B_\rho(x_0)} \frac{|F(\bar{u}(x))-F(\bar{u}(y))|^p}{|x-y|^{n+sp}} \max\{ \bar{\eta}(x), \bar{\eta}(y) \}^p \,\mathrm{d}y \,\mathrm{d}x \\
&\leq C\left( 1+\gamma^{p-1}+\left( \frac{\gamma}{\beta_0} \right)^p \right) \int_{B_r(x_0)} F(\bar{u}(x))^p \int_{B_r(x_0)} \frac{|\bar{\eta}(x)-\bar{\eta}(y)|^p}{|x-y|^{n+sp}} \,\mathrm{d}y\,\mathrm{d}x \\
&\quad + C\left( 1+\gamma^{2sp(p-1)/\varepsilon}+\left( \frac{\gamma}{\beta_0} \right)^p \right) \left( \frac{r}{r-\rho} \right)^p r^{-sp}\|F(\bar{u})\|_{L^p(B_r(x_0))}^p \\
&\quad + C\frac{\gamma^p}{\beta_0} \frac{\mathrm{Tail}^{p-1}(\bar{u}; x_0, r)}{d^{p-1}} \left( \frac{r}{r-\rho} \right)^{n+sp} r^{-sp}\|F(\bar{u})\|_{L^p(B_r(x_0))}^p,
\end{split}
\end{equation}
where $\bar{u}=u_++d$ and $C=C(n, s, p, \Lambda, \varepsilon, R^{\varepsilon}\|b_1\|_{L^{n/(sp-\varepsilon)}(B_R(x_0))})>0$.
\end{lemma}

\begin{proof}
Let us write $B_r = B_r(x_0)$ for simplicity and let $L$ denote the left-hand side of \eqref{eq-Caccio2}. We take a cut-off function $\eta \in C_c^{\infty}(B_{(r+\rho)/2})$ satisfying $\eta = 1$ on $B_\rho$, $0 \leq \eta \leq 1$ and $|\nabla \eta| \leq 2/(r-\rho)$. Since
\begin{equation*}
\begin{split}
&\int_{B_r} F(\bar{u}(x))^p \int_{B_r} \frac{|(\eta\bar{\eta})(x)-(\eta\bar{\eta})(y)|^p}{|x-y|^{n+sp}} \,\mathrm{d}y\,\mathrm{d}x \\
&\leq C \int_{B_r} F(\bar{u}(x))^p \int_{B_r} \left( \frac{|\bar{\eta}(x)-\bar{\eta}(y)|^p}{|x-y|^{n+sp}} + \frac{|\eta(x)-\eta(y)|^p}{|x-y|^{n+sp}} \right) \,\mathrm{d}y\,\mathrm{d}x \\
&\leq C \int_{B_r} F(\bar{u}(x))^p \int_{B_r} \frac{|\bar{\eta}(x)-\bar{\eta}(y)|^p}{|x-y|^{n+sp}} \,\mathrm{d}y\,\mathrm{d}x + C \left( \frac{r}{r-\rho} \right)^p r^{-sp} \|F(\bar{u})\|_{L^p(B_r)}^p,
\end{split}
\end{equation*}
it follows from \eqref{eq-Caccio1} that
\begin{equation*}
\begin{split}
L
&\leq C\left( 1+\gamma^{p-1}+\left( \frac{\gamma}{\beta_0} \right)^p \right) \int_{B_r} F(\bar{u}(x))^p \int_{B_r} \frac{|\bar{\eta}(x)-\bar{\eta}(y)|^p}{|x-y|^{n+sp}} \,\mathrm{d}y\,\mathrm{d}x \\
&\quad + C\left( 1+\gamma^{\frac{2sp(p-1)}{\varepsilon}}+\left( \frac{\gamma}{\beta_0} \right)^p \right) \left( \frac{r}{r-\rho} \right)^p r^{-sp} \|F(\bar{u})\|_{L^p(B_r)}^p \\
&\quad + C \frac{\gamma}{\beta_0} \left( \sup_{x \in B_{(r+\rho)/2}} \int_{\mathbb{R}^n \setminus B_r} \frac{\bar{u}^{p-1}(y)}{|x-y|^{n+sp}} \,\mathrm{d}y \right) \int_{B_r} F(\bar{u}) F'(\bar{u})^{p-1} \,\mathrm{d}x,
\end{split}
\end{equation*}
where $C=C(n, s, p, \Lambda, \varepsilon, R^{\varepsilon}\|b_1\|_{L^{n/(sp-\varepsilon)}(B_R)})>0$.

If $x \in B_{(r+\rho)/2}(x_0)$ and $y \in \mathbb{R}^n \setminus B_r(x_0)$, then $|x-y| \geq \frac{r-\rho}{2r}|y-x_0|$ and hence
\begin{equation*}
\sup_{x \in B_{(r+\rho)/2}} \int_{\mathbb{R}^n \setminus B_r} \frac{\bar{u}^{p-1}(y)}{|x-y|^{n+sp}} \,\mathrm{d}y \leq \frac{C}{r^{sp}} \left( \frac{r}{r-\rho} \right)^{n+sp} \mathrm{Tail}^{p-1}(\bar{u}; x_0, r)
\end{equation*}
for some $C=C(n, s, p)>0$. Moreover, since $F'(t) \leq \frac{\gamma}{p}\frac{F(t)}{d}$, we get that
\begin{equation*}
\int_{B_r} F(\bar{u}) F'(\bar{u})^{p-1} \bar{\eta}^p \,\mathrm{d}x \leq \left( \frac{\gamma}{pd} \right)^{p-1} \int_{B_r} F(\bar{u})^p \,\mathrm{d}x.
\end{equation*}
The desired estimate \eqref{eq-Caccio2} now follows by combining all the estimates above.
\end{proof}

Note that Lemmas~\ref{lem-Caccio1} and \ref{lem-Caccio2} are Moser-type Caccioppoli estimates. We also need the following De Giorgi-type Caccioppoli estimate.

\begin{lemma}\label{lem-Caccio3}
Let $E$ be a relatively closed set in $\Omega$ with measure zero and let $\bar{\eta}$ be a nonnegative smooth function vanishing in a neighborhood of $E$. Let $B_R=B_R(x_0)\subset \Omega$ and let $\eta \in C^\infty_c(B_R)$ be a nonnegative function. If $u$ is a weak subsolution of \eqref{eq-main} in $B_R \setminus E$, then
\begin{equation}\label{eq-Caccio3}
\begin{split}
&\int_{B_R} \int_{B_R} \frac{|w_+(x)-w_+(y)|^p}{|x-y|^{n+sp}} (\eta\bar{\eta})^p(x) \,\mathrm{d}y\,\mathrm{d}x + \int_{B_R} w_+(x) (\eta \bar{\eta})^p(x) \int_{\mathbb{R}^n} \frac{w_-^{p-1}(y)}{|x-y|^{n+sp}} \,\mathrm{d}y \,\mathrm{d}x \\
&\leq C\int_{B_R} w_+^p(x) \int_{B_R} \frac{|(\eta\bar{\eta})(x)-(\eta\bar{\eta})(y)|^p}{|x-y|^{n+sp}} \,\mathrm{d}y \,\mathrm{d}x + C R^{-sp} \|w_+\eta\bar{\eta}\|_{L^p(B_R)}^p \\
&\quad + C \left( R^{\frac{\varepsilon}{p-1}}\|b_2\|^{\frac{p}{p-1}}_{L^{n/(sp-\varepsilon)}(B_R)} + \frac{|k|^p}{R^{\varepsilon}} \|\eta\bar{\eta}\|_{L^\infty(B_R)}^p \right) |A^+(k; x_0, R)|^{1-\frac{sp-\varepsilon}{n}} \\
&\quad + C \left( \sup_{x \in \mathrm{supp} \, \eta} \int_{\mathbb{R}^n \setminus B_R(x_0)} \frac{w_+^{p-1}(y)}{|x-y|^{n+sp}} \,\mathrm{d}y \right) \int_{B_R(x_0)} w_+(x)(\eta\bar{\eta})^p(x) \,\mathrm{d}x,
\end{split}
\end{equation}
where $w_{\pm}=(u-k)_{\pm}$, $A^+(k; x_0, R)=B_R(x_0) \cap \{u>k\}$ and $C$ is a constant depending only on $n$, $s$, $p$, $\Lambda$, $\varepsilon$ and $R^\varepsilon\|b_1\|_{L^{n/(sp-\varepsilon)}(B_R(x_0))}$.
\end{lemma}

\begin{proof}
Let us write $B_R=B_R(x_0)$ and $A^+_R=A^+(k; x_0, R)$ for simplicity. By testing the equation \eqref{eq-main} with $\varphi\coloneqq w_+(\eta\bar{\eta})^p$, we have
\begin{equation}\label{eq-test}
\mathcal{E}(u, \varphi)\leq -\int_{B_R} b(x, u(x)) \varphi(x) \,\mathrm{d}x.
\end{equation}
We basically follow the lines of proof of Proposition~8.5 in Cozzi~\cite{Coz17} to estimate the left-hand side of \eqref{eq-test}, but we keep the cut-off functions $\eta$ and $\bar{\eta}$. Indeed, a slight modification of \cite[(8.4)--(8.6)~and~(8.10)--(8.12)]{Coz17} shows that
\begin{equation*}
\begin{split}
    \mathcal{E}(u, \varphi)
    &\geq C_1 \left( I_1 + \int_{B_R} w_+(x) (\eta \bar{\eta})^p(x) \int_{\mathbb{R}^n} \frac{w_-^{p-1}(y)}{|x-y|^{n+sp}} \,\mathrm{d}y \,\mathrm{d}x \right) \\
    &\quad - C_2 \left( I_2 + \int_{B_R} w_+(x)(\eta\bar{\eta})^p(x) \int_{\mathbb{R}^n \setminus B_R} \frac{w_+^{p-1}(y)}{|x-y|^{n+sp}} \,\mathrm{d}y \,\mathrm{d}x \right)
\end{split}
\end{equation*}
for some positive constants $C_1$ and $C_2$ depending only on $n$, $s$, $p$ and $\Lambda$, where
\begin{equation*}
\begin{split}
I_1 &= \int_{B_R} \int_{B_R} \frac{|w_+(x)-w_+(y)|^p}{|x-y|^{n+sp}} (\eta\bar{\eta})^p(x) \,\mathrm{d}y\,\mathrm{d}x \quad\text{and} \\
I_2 &= \int_{B_R} w_+^p(x) \int_{B_R} \frac{|(\eta\bar{\eta})(x)-(\eta\bar{\eta})(y)|^p}{|x-y|^{n+sp}} \,\mathrm{d}y \,\mathrm{d}x.
\end{split}
\end{equation*}
On the other hand, by observing that $|u| \leq w_++|k|$ on $A^+_R$ and using Young's inequality, we obtain that
\begin{equation*}
\begin{split}
-\int_{B_R} b(x, u(x)) \varphi(x) \,\mathrm{d}x
&\leq \int_{A^+_R} (b_1|u|^{p-1} + b_2) w_+ (\eta\bar{\eta})^p \,\mathrm{d}x \\
&\leq C \int_{A^+_R} (b_1 w_+^{p-1} + b_1 |k|^{p-1} + b_2) w_+ (\eta\bar{\eta})^p \,\mathrm{d}x \\
&\leq C \int_{A^+_R} (b_1 w_+^{p} + b_1 |k|^p + b_2 w_+)(\eta\bar{\eta})^p \,\mathrm{d}x,
\end{split}
\end{equation*}
where $C=C(p)>0$. Moreover, we have that
\begin{equation*}
\int_{A^+_R} b_1 |k|^p (\eta\bar{\eta})^p \,\mathrm{d}x \leq |k|^p \|\eta\bar{\eta}\|_{L^\infty(B_R)}^p |A^+_R|^{1-\frac{sp-\varepsilon}{n}} \|b_1\|_{L^{\frac{n}{sp-\varepsilon}}(B_R)}.
\end{equation*}
Thus, combining all the estimates above shows that
\begin{equation}\label{eq-DG1}
\begin{split}
&I_1 + \int_{B_R} w_+(x) (\eta \bar{\eta})^p(x) \int_{\mathbb{R}^n} \frac{w_-^{p-1}(y)}{|x-y|^{n+sp}} \,\mathrm{d}y \,\mathrm{d}x \\
&\leq CI_2 + C\int_{B_R} w_+(x)(\eta\bar{\eta})^p(x) \int_{\mathbb{R}^n \setminus B_R} \frac{w_+^{p-1}(y)}{|x-y|^{n+sp}} \,\mathrm{d}y \,\mathrm{d}x \\
&\quad + C\frac{|k|^p}{R^{\varepsilon}} \|\eta\bar{\eta}\|_{L^\infty(B_R)}^p |A^+_R|^{1-\frac{sp-\varepsilon}{n}} + C\int_{A^+_R} (b_1 w_+^{p} + b_2 w_+)(\eta\bar{\eta})^p \,\mathrm{d}x
\end{split}
\end{equation}
for some $C=C(n, s, p, \Lambda, R^\varepsilon \|b_1\|_{L^{n/(sp-\varepsilon)}(B_R)})>0$.

Let us now estimate the last integral in \eqref{eq-DG1}. Suppose that $sp < n$ and let $\delta \in (0,1)$. By using H\"older's inequality, Theorem~\ref{thm-Sobolev} and Young's inequality, we have that
\begin{equation*}
\begin{split}
\int_{A^+_R} b_1 (w_+\eta\bar{\eta})^{p} \,\mathrm{d}x
&\leq \|b_1\|_{L^{n/(sp-\varepsilon)}(B_R)} \|w_+\eta\bar{\eta}\|_{L^p(B_R)}^{\varepsilon/s} \|w_+\eta\bar{\eta}\|_{L^{p^{\ast}_{s}}(B_R)}^{p-\varepsilon/s} \\
&\leq C \left( R^{-sp} \|w_+\eta\bar{\eta}\|_{L^p(B_R)}^p \right)^{\frac{\varepsilon}{sp}} \left( [w_+\eta\bar{\eta}]_{W^{s, p}(B_R)}^p + R^{-sp} \|w_+\eta\bar{\eta}\|_{L^p(B_R)}^p \right)^{1-\frac{\varepsilon}{sp}} \\
&\leq \delta [w_+\eta\bar{\eta}]_{W^{s, p}(B_R)}^p + C \delta^{1-\frac{sp}{\varepsilon}} R^{-sp} \|w_+\eta\bar{\eta}\|_{L^p(B_R)}^p
\end{split}
\end{equation*}
for some $C=C(n, s, p, \varepsilon, R^\varepsilon \|b_1\|_{L^{n/(sp-\varepsilon)}(B_R)})>0$. Moreover, we utilize H\"older's inequality, Theorem~\ref{thm-Sobolev} and Young's inequality again, and use
\begin{equation*}
    \left(1-\frac{sp-\varepsilon}{n}-\frac{1}{p_s^{\ast}}\right)\frac{p}{p-1}>1-\frac{sp-\varepsilon}{n}
\end{equation*}
to obtain that
\begin{equation*}
\begin{split}
\int_{A^+_R} b_2 w_+ (\eta\bar{\eta})^p \,\mathrm{d}x
&\leq \|b_2\|_{L^{\frac{n}{sp-\varepsilon}}(B_R)} \|w_+\eta\bar{\eta}\|_{L^{p^{\ast}_{s}}(B_R)}|A^+_R|^{1-\frac{sp-\varepsilon}{n}-\frac{1}{p_s^{\ast}}}\\
&\leq \|b_2\|_{L^{\frac{n}{sp-\varepsilon}}(B_R)} \left( [w_+\eta\bar{\eta}]_{W^{s,p}(B_R)}+R^{-s}\|w_+\eta\bar{\eta}\|_{L^{p}(B_R)}\right) |A^+_R|^{1-\frac{sp-\varepsilon}{n}-\frac{1}{p_s^{\ast}}} \\
&\leq \delta [w_+\eta\bar{\eta}]_{W^{s, p}(B_R)}^p + \delta R^{-sp} \|w_+\eta\bar{\eta}\|^p_{L^{p}(B_R)}\\
&\quad + C\delta^{-\frac{1}{p-1}} \|b_2\|^{\frac{p}{p-1}}_{L^{n/(sp-\varepsilon)}(B_R)} R^{\left(n-sp+\varepsilon-\frac{n}{p_s^{\ast}}\right)\frac{p}{p-1}} \left(\frac{|A^+_R|}{|B_R|}\right)^{1-\frac{sp-\varepsilon}{n}}\\
&\leq \delta [w_+\eta\bar{\eta}]_{W^{s, p}(B_R)}^p + R^{-sp} \|w_+\eta\bar{\eta}\|^p_{L^{p}(B_R)} \\
&\quad+ C\delta^{-\frac{1}{p-1}} R^{\frac{\varepsilon}{p-1}}\|b_2\|^{\frac{p}{p-1}}_{L^{n/(sp-\varepsilon)}(B_R)}|A^+_R|^{1-\frac{sp-\varepsilon}{n}},
\end{split}
\end{equation*}
where $C=C(n, s, p, \varepsilon)>0$. Therefore, we arrive at
\begin{equation*}
\begin{split}
&I_1 + \int_{B_R} w_+(x) (\eta \bar{\eta})^p(x) \int_{\mathbb{R}^n} \frac{w_-^{p-1}(y)}{|x-y|^{n+sp}} \,\mathrm{d}y \,\mathrm{d}x \\
&\leq CI_2 + C \int_{B_R} w_+(x)(\eta\bar{\eta})^p(x) \int_{\mathbb{R}^n \setminus B_R} \frac{w_+^{p-1}(y)}{|x-y|^{n+sp}} \,\mathrm{d}y \,\mathrm{d}x \\
&\quad + 2C\delta [w_+\eta\bar{\eta}]_{W^{s, p}(B_R)}^p + C \left( \delta^{1-\frac{sp}{\varepsilon}} + 1 \right) R^{-sp} \|w_+\eta\bar{\eta}\|_{L^p(B_R)}^p \\
&\quad + C \left( \delta^{-\frac{1}{p-1}} R^{\frac{\varepsilon}{p-1}}\|b_2\|^{\frac{p}{p-1}}_{L^{n/(sp-\varepsilon)}(B_R)} + \frac{|k|^p}{R^{\varepsilon}} \|\eta\bar{\eta}\|_{L^\infty(B_R)}^p \right) |A^+_R|^{1-\frac{sp-\varepsilon}{n}}.
\end{split}
\end{equation*}
By using $[w_+\eta\bar{\eta}]_{W^{s, p}(B_R)}^p \leq 2^{p-1} (I_1+I_2)$ and taking $\delta$ sufficiently small so that $2^{p+1}C\delta<1$, we conclude \eqref{eq-Caccio3} in the subcritical case $sp<n$.

Next, we suppose that $sp=n$. We fix $\tilde{s}$ and $\tilde{p}$ as in \eqref{eq-sp} so that $\tilde{p}<p$ and $\tilde{s}\tilde{p}<n$. As in the subcritical case, we obtain that
\begin{equation*}
\begin{split}
\int_{A^+_R} b_1 (w_+\eta\bar{\eta})^{p} \,\mathrm{d}x
&\leq \|b_1\|_{L^{n/(n-\varepsilon)}(B_R)} \|w_+\eta\bar{\eta}\|_{L^{\tilde{p}}(B_R)}^{\varepsilon/(2\tilde{s})} \|w_+\eta\bar{\eta}\|_{L^{\tilde{p}^{\ast}_{\tilde{s}}}(B_R)}^{p-\varepsilon/(2\tilde{s})} \\
&\leq C \left( R^{-sp} \|w_+\eta\bar{\eta}\|_{L^p(B_R)}^p \right)^{\frac{\varepsilon}{2\tilde{s}p}} \left( [w_+\eta\bar{\eta}]_{W^{s, p}(B_R)}^p + R^{-sp} \|w_+\eta\bar{\eta}\|_{L^p(B_R)}^p \right)^{1-\frac{\varepsilon}{2\tilde{s}p}} \\
&\leq \delta [w_+\eta\bar{\eta}]_{W^{s, p}(B_R)}^p + C \delta^{1-\frac{2\tilde{s}p}{\varepsilon}} R^{-sp} \|w_+\eta\bar{\eta}\|_{L^p(B_R)}^p
\end{split}
\end{equation*}
and that
\begin{equation*}
\begin{split}
\int_{A^+_R} b_2 w_+ (\eta\bar{\eta})^p \,\mathrm{d}x
&\leq \|b_2\|_{L^{\frac{n}{n-\varepsilon}}(B_R)} \|w_+\eta\bar{\eta}\|_{L^{\tilde{p}^{\ast}_{\tilde{s}}}(B_R)}|A^+_R|^{\frac{\varepsilon}{n}-\frac{1}{\tilde{p}_{\tilde{s}}^{\ast}}} \\
&\leq R^{\frac{\varepsilon}{2p}} \|b_2\|_{L^{\frac{n}{n-\varepsilon}}(B_R)} \left( [w_+\eta\bar{\eta}]_{W^{s,p}(B_R)}+R^{-s}\|w_+\eta\bar{\eta}\|_{L^{p}(B_R)}\right) |A^+_R|^{\frac{\varepsilon}{n}-\frac{1}{\tilde{p}_{\tilde{s}}^{\ast}}} \\
&\leq \delta [w_+\eta\bar{\eta}]_{W^{s, p}(B_R)}^p + \delta R^{-sp} \|w_+\eta\bar{\eta}\|^p_{L^{p}(B_R)}\\
&\quad + C\delta^{-\frac{1}{p-1}} \|b_2\|^{\frac{p}{p-1}}_{L^{n/(n-\varepsilon)}(B_R)} R^{\frac{\varepsilon}{2p} \frac{p}{p-1} + \left(\varepsilon-\frac{n-\tilde{s}\tilde{p}}{\tilde{p}}\right)\frac{p}{p-1}} \left(\frac{|A^+_R|}{|B_R|}\right)^{\frac{\varepsilon}{n}}\\
&\leq \delta [w_+\eta\bar{\eta}]_{W^{s, p}(B_R)}^p + R^{-sp} \|w_+\eta\bar{\eta}\|^p_{L^{p}(B_R)} \\
&\quad+ C\delta^{-\frac{1}{p-1}} R^{\frac{\varepsilon}{p-1}}\|b_2\|^{\frac{p}{p-1}}_{L^{n/(n-\varepsilon)}(B_R)}|A^+_R|^{\frac{\varepsilon}{n}}
\end{split}
\end{equation*}
for some $C=C(n, s, p, \varepsilon, R^\varepsilon \|b_1\|_{L^{n/(n-\varepsilon)}(B_R)})>0$, where we also used
\begin{equation*}
\left( \frac{\varepsilon}{n} - \frac{1}{\tilde{p}_{\tilde{s}}^{\ast}} \right) \frac{p}{p-1} > \frac{\varepsilon}{n}.
\end{equation*}
The rest of the proof goes the same way as before.
\end{proof}

By taking $\eta \in C^{\infty}_c(B_{(R+r)/2}(x_0))$ with $\eta=1$ on $B_r(x_0)$, $0 \leq \eta \leq 1$ in $\mathbb{R}^n$ and $|\nabla \eta|\leq 2/(R-r)$, and proceeding as in the proof of Lemma~\ref{lem-Caccio2}, we obtain the following result.

\begin{lemma}\label{lem-Caccio4}
Let $E$ be a relatively closed set in $\Omega$ with measure zero and let $\bar{\eta}$ be a nonnegative smooth function such that $\bar{\eta}$ vanishes in a neighborhood of $E$ and $0 \leq \bar{\eta} \leq 1$. Let $B_r(x_0) \Subset B_R(x_0) \subset \Omega$. If $u$ is a weak subsolution of \eqref{eq-main} in $B_R \setminus E$, then
\begin{equation*}
\begin{split}
&\int_{B_r(x_0)} \int_{B_r(x_0)} \frac{|w_+(x)-w_+(y)|^p}{|x-y|^{n+sp}} \bar{\eta}^p(x) \,\mathrm{d}y\,\mathrm{d}x + \int_{B_r(x_0)} w_+(x) \bar{\eta}^p(x) \int_{\mathbb{R}^n} \frac{w_-^{p-1}(y)}{|x-y|^{n+sp}} \,\mathrm{d}y \,\mathrm{d}x \\
&\leq C\int_{B_R(x_0)} w_+^p(x) \int_{B_R(x_0)} \frac{|\bar{\eta}(x)-\bar{\eta}(y)|^p}{|x-y|^{n+sp}} \,\mathrm{d}y \,\mathrm{d}x + C \left( \frac{R}{R-r} \right)^p R^{-sp} \|w_+\|_{L^p(B_R(x_0))}^p \\
&\quad + C \left( R^{\frac{\varepsilon}{p-1}}\|b_2\|^{\frac{p}{p-1}}_{L^{n/(sp-\varepsilon)}(B_R(x_0))} + \frac{|k|^p}{R^{\varepsilon}} \right) |A^+(k; x_0, R)|^{1-\frac{sp-\varepsilon}{n}} \\
&\quad + C \left( \frac{R}{R-r} \right)^{n+sp} R^{-sp} \|w_+\|_{L^1(B_R(x_0))} \mathrm{Tail}^{p-1}(w_+; x_0, R),
\end{split}
\end{equation*}
where $w_\pm=(u-k)_\pm$, $A^+(k; x_0, R)=B_R(x_0) \cap \{u>k\}$, and $C$ is a constant depending only on $n$, $s$, $p$, $\Lambda$, $\varepsilon$, $R^\varepsilon\|b_1\|_{L^{n/(sp-\varepsilon)}(B_R(x_0))}$.
\end{lemma}

If $u$ is a weak supersolution of \eqref{eq-main} in $B_R \setminus E$, where $E$ is given as in Lemma~\ref{lem-Caccio4}, then $v=-u$ is a weak subsolution of $\mathcal{L}v+\tilde{b}(x, v)=0$ in $B_R \setminus E$ with $\tilde{b}(x, z)=-b(x, -z)$. Since $\tilde{b}$ satisfies the same structure condition
\begin{equation*}
    |\tilde{b}(x, z)| \leq b_1(x) |z|^{p-1} + b_2(x),
\end{equation*}
an application of Lemma~\ref{lem-Caccio4} to $v$ shows that
\begin{equation}\label{eq-Caccio-super}
\begin{split}
&\int_{B_r(x_0)} \int_{B_r(x_0)} \frac{|w_-(x)-w_-(y)|^p}{|x-y|^{n+sp}} \bar{\eta}^p(x) \,\mathrm{d}y\,\mathrm{d}x + \int_{B_r(x_0)} w_-(x) \bar{\eta}^p(x) \int_{\mathbb{R}^n} \frac{w_+^{p-1}(y)}{|x-y|^{n+sp}} \,\mathrm{d}y \,\mathrm{d}x \\
&\leq C\int_{B_R(x_0)} w_-^p(x) \int_{B_R(x_0)} \frac{|\bar{\eta}(x)-\bar{\eta}(y)|^p}{|x-y|^{n+sp}} \,\mathrm{d}y \,\mathrm{d}x + C \left( \frac{R}{R-r} \right)^p R^{-sp} \|w_-\|_{L^p(B_R(x_0))}^p \\
&\quad + C \left( R^{\frac{\varepsilon}{p-1}}\|b_2\|^{\frac{p}{p-1}}_{L^{n/(sp-\varepsilon)}(B_R(x_0))} + \frac{|k|^p}{R^{\varepsilon}} \right) |A^-(k; x_0, R)|^{1-\frac{sp-\varepsilon}{n}} \\
&\quad + C \left( \frac{R}{R-r} \right)^{n+sp} R^{-sp} \|w_-\|_{L^1(B_R(x_0))} \mathrm{Tail}^{p-1}(w_-; x_0, R),
\end{split}
\end{equation}
where $A^-(k; x_0, R)=B_R(x_0) \cap \{u<k\}$.

When $E=\emptyset$ in Lemma~\ref{lem-Caccio3}, we obtain the following standard Caccioppoli estimate. It shows that any weak subsolution of \eqref{eq-main} belongs to the so-called (fractional) De Giorgi class, cf.\ Section~6.1 in Cozzi~\cite{Coz17}. Note that such Caccioppoli estimates were obtained for equations similar to \eqref{eq-main} in \cite[Proposition~8.5]{Coz17}, but the equation \eqref{eq-main} was not covered.

\begin{lemma}\label{lem-DG}
If $u$ is a weak subsolution of \eqref{eq-main} in $\Omega$, then for any $B_r(x_0) \Subset B_R(x_0) \subset \Omega$ and $k \in \mathbb{R}$,
\begin{equation*}
\begin{split}
&[w_+]_{W^{s, p}(B_r(x_0))}^p + \int_{B_r(x_0)} w_+(x) \int_{\mathbb{R}^n} \frac{w_-^{p-1}(y)}{|x-y|^{n+sp}} \,\mathrm{d}y \,\mathrm{d}x \\
&\leq + C \left( \frac{R}{R-r} \right)^p R^{-sp} \|w_+\|_{L^p(B_R(x_0))}^p \\
&\quad + C \left(R^{\frac{\varepsilon}{p-1}}\|b_2\|^{\frac{p}{p-1}}_{L^{n/(sp-\varepsilon)}(B_R(x_0))} + \frac{|k|^p}{R^{\varepsilon}} \right) |A^+(k; x_0, R)|^{1-\frac{sp-\varepsilon}{n}} \\
&\quad + C \left( \frac{R}{R-r} \right)^{n+sp} R^{-sp} \|w_+\|_{L^1(B_R(x_0))} \mathrm{Tail}^{p-1}(w_+; x_0, R)
\end{split}
\end{equation*}
where $w_{\pm}=(u-k)_{\pm}$, $A^+(k; x_0, R)=B_R(x_0) \cap \{u>k\}$ and $C>0$ is a constant depending only on $n$, $s$, $p$, $\Lambda$, $\varepsilon$ and $R^\varepsilon\|b_1\|_{L^{n/(sp-\varepsilon)}(B_R(x_0))}$.
\end{lemma}

\begin{proof}
The proof goes as in those of Lemmas~\ref{lem-Caccio3} and \ref{lem-Caccio4}, with the only difference that we take $\bar{\eta}\equiv 1$.
\end{proof}

\subsection{Interior regularity estimates}

Interior regularity estimates such as H\"older estimates and the Harnack inequality follow directly from the Caccioppoli estimate in Lemma~\ref{lem-DG} and the general regularity results for functions in the De Giorgi class.

\begin{theorem}\label{thm-holder}
Assume that $\Omega$ is bounded. If $u$ is a weak solution of \eqref{eq-main} in $\Omega$, then $u \in C^\alpha_{\mathrm{loc}}(\Omega)$ for some $\alpha \in (0,1)$.
\end{theorem}

\begin{proof}
Lemma~\ref{lem-DG} shows that
\begin{equation}\label{eq-u-DG}
u \in \mathrm{DG}^{s, p}\left(\Omega; \|b_2\|^{1/(p-1)}_{L^{n/(sp-\varepsilon)}(\Omega)}, C, -\infty, \frac{\varepsilon}{n}, \frac{\varepsilon}{p-1}, \infty \right)
\end{equation}
for some $C=C(n, s, p, \Lambda, \varepsilon, (\mathrm{diam}\,\Omega)^\varepsilon \|b_1\|_{L^{n/(sp-\varepsilon)}(\Omega)})>0$, where $\mathrm{DG}^{s, p}$ is the fractional De Giorgi class given in \cite[Section~6.1]{Coz17}. Thus, the desired result follows from \cite[Theorem~6.4]{Coz17}.
\end{proof}

Another important regularity estimate we need in this work is the Harnack inequality.

\begin{theorem}[Harnack inequality]\label{thm-harnack}
Let $u$ be a weak solution of \eqref{eq-main} in $B_R=B_R(x_0)$ such that $u \geq 0$ in $B_R$. Then for any $r \leq R/2$,
	\begin{equation*}
		\sup_{B_r}u \leq C \inf_{B_r}u + C \mathrm{Tail}(u_-; x_0, r) + C\left( R^\varepsilon \|b_2\|_{L^{n/(sp-\varepsilon)}(B_R)} \right)^{1/(p-1)},
	\end{equation*}
where $C=C(n, s, p, \Lambda, \varepsilon, R^\varepsilon\|b_1\|_{L^{n/(sp-\varepsilon)}(B_R)})>0$.
\end{theorem}

\begin{proof}
Since \eqref{eq-u-DG} holds with $\Omega=B_R$, the desired result follows from \cite[Theorem~6.9]{Coz17}.
\end{proof}

The following form of the Harnack inequality on annuli can be obtained by using a standard covering argument, but we provide a proof for completeness.

\begin{theorem}\label{thm-Harnack-annulus}
Assume $n \geq 2$. Let $u$ be a weak solution of \eqref{eq-main} in $B_R \setminus \{0\}$ such that $u \geq 0$ in $B_R \setminus \{0\}$. Then for any $r \leq R/2$, 
\begin{equation*}
\sup_{B_r \setminus B_{r/2}}{u} \leq C \inf_{B_r \setminus B_{r/2}}{u} + C \left(\frac{r}{R}\right)^{\frac{sp}{p-1}}\mathrm{Tail}(u_-; 0, R) + C\left(\frac{r}{R}\right)^{\frac{\varepsilon}{p-1}}\left( R^\varepsilon \|b_2\|_{L^{n/(sp-\varepsilon)}(B_R)} \right)^{\frac{1}{p-1}},
\end{equation*}
where $C=C(n, s, p, \Lambda, \varepsilon, R^\varepsilon\|b_1\|_{L^{n/(sp-\varepsilon)}(B_R)})>0$.
\end{theorem}

\begin{proof}
We can cover $B_r \setminus B_{r/2}$ by a finite number of open balls $\{B_{r/4}(y_i)\}_{i=1}^N$ with $y_i \in A_r$ such that $0 \notin B_{r/2}(y_i)$. Note that $N$ depends only on $n$. An application of Theorem~\ref{thm-harnack} in $B_{r/4}(y_i) \subset B_{r/2}(y_i)$ shows that
\begin{equation*}
\sup_{B_{r/4}(y_i)}{u} \leq C \inf_{B_{r/4}(y_i)}{u} + C\, \mathrm{Tail}(u_-; y_i, r/4) + C \left( r^\varepsilon \|b_2\|_{L^{n/(sp-\varepsilon)}(B_{r/2}(y_i))} \right)^{1/(p-1)}.
\end{equation*}
Note that the constant $C$ is independent of $i$. Since $u \geq 0$ in $B_R \setminus \{0\}$ and
\begin{equation*}
|y| \leq |y-y_i|+|y_i| \leq |y-y_i|+r \leq 2|y-y_i| \quad\text{for all}~y \in B_R^c,
\end{equation*}
we get that
\begin{equation*}
\mathrm{Tail}^{p-1}(u_-; y_i, r/4) = \left( \frac{r}{4} \right)^{sp} \int_{\mathbb{R}^n \setminus B_R} \frac{u_-^{p-1}(y)}{|y-y_i|^{n+sp}} \,\mathrm{d}y \leq C \left( \frac{r}{R} \right)^{sp} \mathrm{Tail}^{p-1}(u_-; 0, R).
\end{equation*}
Let $j \in \{1, \dots, N\}$ be such that
\begin{equation*}
\inf_{B_r \setminus B_{r/2}}{u} \geq \inf_{B_{r/4}(y_j)}{u}.
\end{equation*}
Then, since $B_r \setminus B_{r/2}$ is connected, we obtain that
\begin{equation*}
\begin{split}
\sup_{B_r \setminus B_{r/2}}u \leq \max_{1\leq i \leq N} \sup_{B_{r/4}(y_i)}u
&\leq C^N \inf_{B_{r/4}(y_j)} {u} + C(N) \left( \frac{r}{R} \right)^{sp} \mathrm{Tail}^{p-1}(u_-; 0, R) \\
&\quad + C(N) \left(\frac{r}{R}\right)^{\frac{\varepsilon}{p-1}}\left( R^\varepsilon \|b_2\|_{L^{n/(sp-\varepsilon)}(B_R)} \right)^{1/(p-1)},
\end{split}
\end{equation*}
as desired.
\end{proof}

\section{Removable singularity}\label{sec-removable}

The aim of this section is to prove Theorem~\ref{thm-rem} and Corollary~\ref{cor-rem} concerning removable singularities.

\begin{proof}[Proof of Theorem~\ref{thm-rem}]
If $tq>n$, then it follows from Lemma~\ref{lem-cap-point} that $E$ is an empty set. Thus, we assume in the rest of the proof that $tq \leq n$. Note that $sp \leq tq \leq n$.

Once we prove that $u$ is a weak solution of \eqref{eq-main} in $\Omega$, then Theorem~\ref{thm-holder} shows that $u$ has a representative which is continuous in $\Omega$. As being a weak solution is a local property, it is enough to prove that each point $x_0 \in \Omega \cap E$ has a neighborhood in which $u$ is a weak solution. Fix $x_0 \in \Omega \cap E$ and take $R>0$ sufficiently small so that $B_{R}=B_{R}(x_0) \Subset \Omega$. We will show that $u$ is a weak solution of \eqref{eq-main} in $B_{R/4}$.

We first prove that $u \in W^{s, p}(B_{R/2})$. Since $E$ has $(t, q)$-capacity zero, so is $E \cap \overline{B}_R$. Since $E \cap \overline{B}_R$ is compact, by \eqref{eq-cap-E}, for each $j\in \mathbb{N}$, there are open sets $G_j$ and $G_j'$ such that
\begin{equation*}
E \cap \overline{B}_R \Subset G_j \Subset G_j' \Subset B_{2R} \quad\text{and}\quad \mathrm{cap}_{t, q}(G_j', B_{2R}) < 1/j.
\end{equation*}
Then $\mathrm{cap}_{t, q}(\overline{G}_j, B_{2R})<1/j$, and hence by definition there are $\eta_j' \in C^\infty_c(B_{2R})$ such that $\eta_j' \geq 1$ on $\overline{G}_j$ and $[\eta_j']_{W^{t, q}(\mathbb{R}^n)}<1/j$. Thus, there is a mollification $\eta_j \in C^\infty_c(B_{2R})$ of $\min\{\eta_j', 1\}_+$ such that $0 \leq \eta_j \leq 1$ in $\mathbb{R}^n$, $\eta_j=1$ in a neighborhood of $E \cap \overline{B}_R$ and $[\eta_j]_{W^{t, q}(\mathbb{R}^n)}<1/j$. We may also assume that $\eta_j \to 0$ a.e.\ in $\mathbb{R}^n$ after taking a subsequence if necessary. Let $\bar{\eta}_j=1-\eta_j$. Then $\bar{\eta}_j$ is a nonnegative smooth function vanishing in a neighborhood of $E \cap \overline{B}_R$ such that $[\bar{\eta}_j]_{W^{t, q}(\mathbb{R}^n)} \to 0$ and $\bar{\eta}_j \to 1$ a.e.\ in $\mathbb{R}^n$ as $j \to \infty$.

Let $R/2\leq \rho<r\leq R$, $\beta_0=\delta(p-1)>0$ and
\begin{equation*}
d=\left( R^{\varepsilon}\|b_2\|_{L^{n/(sp-\varepsilon)}(B_{R})} \right)^{1/(p-1)}+\mathrm{Tail}(u_+; x_0, R/2).
\end{equation*}
Note that $|E \cap \overline{B}_R|=0$ by Lemma~\ref{lem-measurezero}. Thus, Lemma~\ref{lem-Caccio2} with $E$ and $\bar{\eta}$ replaced by $E \cap \overline{B}_R$ and $\bar{\eta}_j$, respectively, shows that
\begin{equation*}
\begin{split}
L&\coloneqq \int_{B_\rho} \int_{B_\rho} \frac{|F(\bar{u}(x))-F(\bar{u}(y))|^p}{|x-y|^{n+sp}} \max\{ \bar{\eta}_j(x), \bar{\eta}_j(y) \}^p \,\mathrm{d}y \,\mathrm{d}x \\
&\leq C(1+\gamma^p) \int_{B_r} F(\bar{u}(x))^p \int_{B_r} \frac{|\bar{\eta}_j(x)-\bar{\eta}_j(y)|^p}{|x-y|^{n+sp}} \,\mathrm{d}y\,\mathrm{d}x \\
&\quad + \frac{C}{R^{sp}} \left( 1+\gamma^{2sp(p-1)/\varepsilon} + \gamma^p \right) \left( \frac{r}{r-\rho} \right)^{n+p} \left( 1+\frac{\mathrm{Tail}^{p-1}(\bar{u}; x_0, R/2)}{d^{p-1}} \right) \|F(\bar{u})\|_{L^p(B_r)}^p \\
&\eqqcolon I_1 + I_2,
\end{split}
\end{equation*}
where $\bar{u}=u_++d$ and $C=C(n, s, p, \Lambda, \delta, \varepsilon, R^\varepsilon\|b_1\|_{L^{n/(sp-\varepsilon)}(B_{R})})>0$.

Since $F(\bar{u})^p \leq C(p, \gamma, \gamma_0, l) \bar{u}^{\gamma_0}$ and $\gamma_0 = (1+\delta)(p-1)$, it follows from Lemma~\ref{lem-Holder} that
\begin{equation*}
I_1 \leq C r^{(t-s)p} \|\bar{u}\|_{L^{(1+\delta)\theta}(B_r)}^{\gamma_0} [\bar{\eta}_j]_{W^{t, q}(B_r)}^p,
\end{equation*}
where the constant $C$ is independent of $j$. Since $u \in L^{(1+\delta)\theta}(\Omega)$ and $[\bar{\eta}_j]_{W^{t, q}(\mathbb{R}^n)} \to 0$ as $j \to \infty$, we obtain that $\lim_{j \to \infty} {I_1}=0$. Therefore, Fatou's lemma shows that
\begin{equation}\label{eq-I2}
[F(\bar{u})]_{W^{s, p}(B_\rho)}^p \leq \liminf_{j \to \infty} {L} \leq \liminf_{j \to \infty} {(I_1+I_2)} = I_2.
\end{equation}
Let
\begin{equation*}
\chi=
\begin{cases}
n/(n-sp) &\text{if}~sp<n, \\
\text{any number larger than 1} &\text{if}~sp=n.
\end{cases}
\end{equation*}
If $sp < n$, then the estimate \eqref{eq-I2} together with the fractional Sobolev inequality (Theorem~\ref{thm-Sobolev}) yields
\begin{equation}\label{eq-Sobolev}
\begin{split}
\left( \fint_{B_\rho} F(\bar{u})^{p\chi} \,\mathrm{d}x \right)^{1/\chi}
&\leq C \fint_{B_\rho} F(\bar{u})^p \,\mathrm{d}x + C\rho^{sp-n} [F(\bar{u})]_{W^{s, p}(B_\rho)}^p \\
&\leq C\left(1+\gamma^{2sp(p-1)/\varepsilon} + \gamma^p \right) \left( \frac{r}{r-\rho} \right)^{n+p} \fint_{B_r} F(\bar{u})^p \,\mathrm{d}x,
\end{split}
\end{equation}
where $C=C(n, s, p, \Lambda, \delta, \varepsilon, R^\varepsilon \|b_1\|_{L^{n/(sp-\varepsilon)}(B_R)})>0$. If $sp=n$, then we take $\tilde{p} \in (1, p)$ and $\tilde{s} \in (0, s)$ so that $\tilde{s} \tilde{p} < n$ and $\chi=\tilde{p}_{\tilde{s}}^\ast/p$, where $\tilde{p}_{\tilde{s}}^\ast = n\tilde{p}/(n-\tilde{s} \tilde{p})$. Then the same estimate \eqref{eq-Sobolev} follows from the fractional Sobolev inequality with $\tilde{s}\tilde{p}<n$, Lemma~\ref{lem-Jensen}, Jensen's inequality and the estimate \eqref{eq-I2}. In this case, the constant $C$ depends also on $\tilde{s}$ and $\tilde{p}$.

Recall the definition of $F$ given in \eqref{eq-F}. By passing to the limit as $l\to\infty$, we obtain that
\begin{equation}\label{eq-Phi-gamma}
\Phi(\gamma \chi, \rho) \leq \left[ C\left( 1 + \gamma^{2sp(p-1)/\varepsilon} + \gamma^p \right) \left( \frac{r}{r-\rho} \right)^{n+p} \right]^{1/\gamma} \Phi(\gamma, r),
\end{equation}
where
\begin{equation*}
\Phi(\gamma, r) = \left( \fint_{B_r} \bar{u}^{\gamma} \,\mathrm{d}x \right)^{1/\gamma}.
\end{equation*}
Note that \eqref{eq-Phi-gamma} holds for any $\gamma \geq \gamma_0$.

The standard iteration argument shows that
\begin{equation}\label{eq-bdd}
\begin{split}
\|\bar{u}\|_{L^\infty(B_{R/2})}
&\leq C \Phi(\gamma_0, R) \leq C \Phi((1+\delta)\theta, R) \\
&\leq C R^{-\frac{n}{(1+\delta)\theta}} \|u_+\|_{L^{(1+\delta)\theta}(B_{R})} + \mathrm{Tail}(u_+; x_0, R/2) < \infty.
\end{split}
\end{equation}
Moreover, it follows from \eqref{eq-I2} with $\gamma=\gamma_0$ and $\rho=R/2$ that
\begin{equation}\label{eq-seminorm}
\int_{B_{R/2}} \int_{B_{R/2}} \frac{|\bar{u}^{\gamma_0/p}(x) - \bar{u}^{\gamma_0/p}(y)|^p}{|x-y|^{n+sp}} \,\mathrm{d}y \,\mathrm{d}x < \infty.
\end{equation}
Since
\begin{equation*}
|a-b|^p \min\{ a^{\gamma_0-p}, b^{\gamma_0-p} \} \leq \left| \int_b^a t^{\gamma_0/p-1} \,\mathrm{d}t \right|^p = \left( \frac{p}{\gamma_0} \right)^p \left|a^{\gamma_0/p} - b^{\gamma_0/p} \right|^p \quad\text{for any }a, b \geq d,
\end{equation*}
we obtain from \eqref{eq-seminorm} that
\begin{equation}\label{eq-seminorm2}
\int_{B_{R/2}} \int_{B_{R/2}} \frac{|\bar{u}(x) - \bar{u}(y)|^p}{|x-y|^{n+sp}} \min\{ \bar{u}^{\gamma_0-p}(x), \bar{u}^{\gamma_0-p}(y) \}\,\mathrm{d}y \,\mathrm{d}x < \infty.
\end{equation}
We may assume that $\delta>0$ is sufficiently small so that $\gamma_0=(1+\delta)(p-1) \leq p$. Then the estimate \eqref{eq-seminorm2} together with the boundedness \eqref{eq-bdd} yields $[u_+]_{W^{s, p}(B_{R/2})} < \infty$. Note that we have only used the fact that $u$ is a weak subsolution of \eqref{eq-main} in $\Omega \setminus E$. By using that $u$ is a weak supersolution of \eqref{eq-main} in $\Omega \setminus E$, one can also shows that $[u_-]_{W^{s, p}(B_{R/2})} < \infty$. This concludes that $u \in W^{s, p}(B_{R/2})$.

We next prove that $u$ is indeed a weak solution of \eqref{eq-main} in $B_{R/4}$. If we let $\varphi \in C^\infty_c(B_{R/4})$, then $\varphi \bar{\eta}_j \in C^\infty_c(B_{R/4} \setminus E)$. Since $u$ is a weak solution of \eqref{eq-main} in $\Omega \setminus E$, we have
\begin{align*}
0
&=\mathcal{E}(u, \varphi\bar{\eta}_j) + \int_{B_{R/4}} b(x, u(x))\varphi(x)\bar{\eta}_j(x)\,\mathrm{d}x \\
&= \int_{B_{R/2}} \int_{B_{R/2}} |u(x)-u(y)|^{p-2}(u(x)-u(y))(\bar{\eta}_j(x)-\bar{\eta}_j(y))\varphi(y) k(x, y) \,\mathrm{d}y\,\mathrm{d}x \\
&\quad + \int_{B_{R/2}} \int_{B_{R/2}} |u(x)-u(y)|^{p-2}(u(x)-u(y))(\varphi(x)-\varphi(y))\bar{\eta}_j(x) k(x, y) \,\mathrm{d}y\,\mathrm{d}x \\
&\quad + 2\int_{B_{R/2}} \int_{\mathbb{R}^n \setminus B_{R/2}} |u(x)-u(y)|^{p-2}(u(x)-u(y)) \varphi(x)\bar{\eta}_j(x) k(x, y) \,\mathrm{d}y\,\mathrm{d}x \\
&\quad + \int_{B_{R/4}} b(x, u(x))\varphi(x)\bar{\eta}_j(x)\,\mathrm{d}x \\
&\eqqcolon J_1 + J_2 + J_3 + J_4.
\end{align*}
If $q>p$, then it follows from H\"older's inequality and Lemma~\ref{lem-Jensen} that
\begin{equation}\label{eq-J1}
|J_1| \leq CR^{n\frac{q-p}{pq}+t-s} \|\varphi\|_{L^\infty(B_{R/4})} [u]_{W^{s, p}(B_{R/2})}^{p-1} [\bar{\eta}_j]_{W^{t, q}(B_{R/2})}.
\end{equation}
If $q=p$, then \eqref{eq-J1} is nothing but H\"older's inequality. The estimate \eqref{eq-J1} together with $u \in W^{s ,p}(B_{R/2})$ and $\lim_{j\to\infty} [\bar{\eta}_j]_{W^{t, q}(\mathbb{R}^n)} = 0$ implies that $\lim_{j \to \infty} J_1=0$.

We observe that the integrands in $J_2$ and $J_3$ are bounded by functions
\begin{align*}
H_1 &\coloneqq \frac{|u(x)-u(y)|^p}{|x-y|^{n+sp}} + \frac{|\varphi(x)-\varphi(y)|^p}{|x-y|^{n+sp}} \quad\text{and} \\
H_2 &\coloneqq \frac{|u(x)|^{p-1} + |u(y)|^{p-1}}{|x-y|^{n+sp}} |\varphi(x)|,
\end{align*}
respectively, up to constants depending only on $p$ and $\Lambda$. Clearly, $H_1 \in L^1(B_{R/2} \times B_{R/2})$. Moreover, we see by using Lemma~\ref{lem-tail} that
\begin{align*}
\|H_2\|_{L^1(B_{R/2} \times B_{R/2}^c)}
&\leq \|\varphi\|_{L^\infty(B_{R/4})} \int_{B_{R/4}} \int_{B_{R/2}^c} \frac{|u(x)|^{p-1}+|u(y)|^{p-1}}{|x-y|^{n+sp}}\,\mathrm{d}y\,\mathrm{d}x\\
&\leq \frac{C}{R^{sp}} \|\varphi\|_{L^\infty(B_{R/4})} \left( \|u\|_{L^{p-1}(B_{R/4})}^{p-1} + R^n \mathrm{Tail}^{p-1}(u; x_0, R/2) \right) < \infty.
\end{align*}
It therefore follows from the dominated convergence theorem that 
\begin{equation*}
    \mathcal{E}(u, \varphi)+\int_{B_{R/4}} b(x, u(x))\varphi(x)\,\mathrm{d}x=0
\end{equation*} 
as desired.
\end{proof}

\begin{proof}[Proof of Corollary~\ref{cor-rem}]
We take $q \in (p, n/s)$ sufficiently close to $n/s$ so that
		\begin{equation*}
			((n-sp)-\delta(p-1)) \frac{q}{q-p}<n.
		\end{equation*}
		Then there exists $\delta_0>0$ such that
		\begin{equation*}
			-\left(\frac{n-sp}{p-1}-\delta\right)(1+\delta_0) \theta=-(1+\delta_0)	\left((n-sp)-\delta(p-1)\right) \frac{q}{q-p}>-n,
		\end{equation*}
		where $\theta=q(p-1)/(q-p)$. This implies that $u \in L^{(1+\delta_0)\theta}(\Omega)$. We now choose $t \in (s, 1)$ sufficiently close to $s$ so that $tq<n$, then $\{0\}$ is of $(t, q)$-capacity zero by Lemma~\ref{lem-cap-point}. Thus, Theorem~\ref{thm-rem} shows that the origin is a removable singularity.
\end{proof}

\section{Isolated singularity}\label{sec-isolated}

The goal of this section is to describe the asymptotic behavior of a singular solution near an isolated non-removable singularity. Namely, we provide the proofs of Theorems~\ref{thm-iso-sub} and \ref{thm-iso-log}. \emph{In the rest of the paper, we always assume that $\Omega$ contains the origin and that $\Omega$ is bounded as Theorems~\ref{thm-iso-sub} and \ref{thm-iso-log} focus on the isolated singularity at the origin. Note, however, that these theorems still hold for unbounded $\Omega$. Moreover, we may assume that $u \in C(\overline{\Omega} \setminus \{0\}) \cap W^{s,p}(\Omega \setminus B_{\rho})$ for all $\rho>0$ by recalling Theorem~\ref{thm-holder} and choosing $\Omega' \Subset \Omega$ if necessary.} The last assumption allows us to avoid technical issues arising from the regularity of $u$ near $\partial \Omega$.

We begin with the following result.

\begin{lemma}[Non-removable singularity]\label{lem-non-rem}
 Let $u$ be a weak solution of \eqref{eq-main} in $\Omega \setminus \{0\}$ such that $u$ is positive in $\Omega \setminus \{0\}$. If $u$ has a non-removable singularity at the origin, then
	\begin{equation*}
		\lim_{x \to 0}u(x)=\infty.
	\end{equation*}
\end{lemma}

For local equations, such a limit behavior follows rather easily from the removable singularity theorem, the Harnack inequality and the minimum principle, see Section~7 in Serrin~\cite{Ser64} for instance. However, the (global) minimum principle is not suitable for capturing this phenomenon in the nonlocal case. We thus provide an alternative approach based on the following weak Harnack inequality.

\begin{lemma}[Weak Harnack inequality]\label{lem-WHI}
Let $u$ be a weak supersolution of \eqref{eq-main} in $B_R \setminus \{0\}$ such that $u$ is nonnegative in $B_R \setminus \{0\}$. Then for any $r \leq R/2$,
\begin{equation*}
	\left(\fint_{B_r} u^{p_0}\,\mathrm{d}x\right)^{\frac{1}{p_0}} \leq C \inf_{B_{r}} u+C \mathrm{Tail}(u_-; 0, R) + C\left( R^\varepsilon \|b_2\|_{L^{\frac{n}{sp-\varepsilon}}(B_R)} \right)^{\frac{1}{p-1}},
\end{equation*}
where $p_0 \in (0,1)$ and $C>0$ depend only on $n$, $s$, $p$, $\Lambda$, $\varepsilon$ and $R^\varepsilon \|b_1\|_{L^{n/(sp-\varepsilon)}(B_R)}$.
\end{lemma}

\begin{proof}
Since $\{0\}$ is of $(s, p)$-capacity zero, there exists a sequence of smooth functions $\bar{\eta}_j$ such that $0\leq \bar{\eta}_j\leq 1$ in $\mathbb{R}^n$, $\lim_{j \to \infty} [\bar{\eta}_j]_{W^{s, p}(\mathbb{R}^n)}=0$, $\bar{\eta}_j \to 1$ a.e.\ in $\mathbb{R}^n$ as $j \to \infty$ and $\bar{\eta}_j$ vanishes in a neighborhood of $\{0\}$; see the beginning of the proof of Theorem~\ref{thm-rem}. We now set $E=\{0\}$, let $B_{r_1}(x_0) \Subset B_{r_2}(x_0) \subset B_R$ and take $\bar{\eta}=\bar{\eta}_j$ in \eqref{eq-Caccio-super}. Since $w_-\coloneqq (u-k)_- \in L^\infty(B_R)$ for all $k \in \mathbb{R}$ by assumption, it follows from Fatou's lemma that
\begin{equation*}
\begin{split}
&\int_{B_{r_1}(x_0)} \int_{B_{r_1}(x_0)} \frac{|w_-(x)-w_-(y)|^p}{|x-y|^{n+sp}} \,\mathrm{d}y\,\mathrm{d}x + \int_{B_{r_1}(x_0)} w_-(x) \int_{\mathbb{R}^n} \frac{w_+^{p-1}(y)}{|x-y|^{n+sp}} \,\mathrm{d}y \,\mathrm{d}x \\
&\leq C\|w_-\|_{L^\infty(B_{r_2}(x_0))}^p \lim_{j \to \infty} [\bar{\eta}_j]_{W^{s, p}(B_{r_2}(x_0))}^p + C r_2^{-sp} \left( \frac{r_2}{r_2-r_1} \right)^p \|w_-\|_{L^p(B_{r_2}(x_0))}^p \\
&\quad + C \left( r_2^{\frac{\varepsilon}{p-1}}\|b_2\|^{\frac{p}{p-1}}_{L^{n/(sp-\varepsilon)}(B_R)} + \frac{|k|^p}{r_2^{\varepsilon}} \right) |A^-(k; x_0, r_2)|^{1-\frac{sp-\varepsilon}{n}} \\
&\quad + C r_2^{-sp} \left( \frac{r_2}{r_2-r_1} \right)^{n+sp} \|w_-\|_{L^1(B_{r_2}(x_0))} \mathrm{Tail}^{p-1}(w_-; x_0, r_2).
\end{split}
\end{equation*}
Since $\lim_{j \to \infty} [\bar{\eta}_j]_{W^{s, p}(\mathbb{R}^n)}=0$, it implies that
\begin{equation*}
    u \in \mathrm{DG}_-^{s, p}\left(\Omega; \|b_2\|^{1/(p-1)}_{L^{n/(sp-\varepsilon)}(B_R)}, C, -\infty, \frac{\varepsilon}{n}, \frac{\varepsilon}{p-1}, \infty \right)
\end{equation*}
for some $C=C(n, s, p, \Lambda, \varepsilon, R^\varepsilon \|b_1\|_{L^{n/(sp-\varepsilon)}(B_R)})>0$, where $\mathrm{DG}_-^{s,p}$ is the fractional De Giorgi class given in \cite[Section~6.1]{Coz17}. Thus, the desired estimate follows from \cite[Proposition~6.8]{Coz17}.
\end{proof}

Let us now prove Lemma~\ref{lem-non-rem} by using the weak Harnack inequality given in Lemma~\ref{lem-WHI}.

\begin{proof}[Proof of Lemma~\ref{lem-non-rem}]
Since $sp \leq n$, it follows from Lemma~\ref{lem-cap-point} that $\{0\}$ is of $(s,p)$-capacity zero. Thus, Theorem~\ref{thm-rem} implies that
\begin{equation}\label{eq-limsup}
	\limsup_{x \to 0}u(x)=\infty.
\end{equation}
It only remains to prove that $\liminf_{x \to 0}u(x)=\infty$. Suppose by contradiction that 
\begin{equation*}
	\liminf_{x \to 0}u(x)=L<\infty.
\end{equation*}
Let $B_R=B_R(0) \subset \Omega$. Then, for each $M>0$, there exists a radius $r=r_M \in (0, R/2)$ such that
\begin{equation*}
u \geq M
\end{equation*}
in $B_r \setminus B_{r/2}$ for $n \geq 2$ and in $(r/2, r)$ or in $(-r, -r/2)$ for $n=1$, by the Harnack inequality (Theorem~\ref{thm-Harnack-annulus} for $n \geq 2$ and Theorem~\ref{thm-harnack} for $n=1$) and \eqref{eq-limsup}. By using Lemma~\ref{lem-WHI} and the fact that $u \geq 0$ in $B_R \setminus \{0\}$, we obtain that
\begin{equation*}
	\left(\fint_{B_{r}} u^{p_0}\,\mathrm{d}x\right)^{1/p_0} \leq C \inf_{B_{r}} u+C\mathrm{Tail}(u_-; 0, R) + C \left( R^\varepsilon \|b_2\|_{L^{n/(sp-\varepsilon)}(B_R)} \right)^{1/(p-1)}
\end{equation*}
for some $p_0 \in (0,1)$. By combining all the estimates above and taking sufficiently small $r$, we conclude that $M \leq CL+C$ for some $C>0$, which leads to the contradiction.
\end{proof}

In order to prove the upper bounds in Theorems~\ref{thm-iso-sub} and \ref{thm-iso-log}, we first introduce the following invariant.

\begin{lemma}\label{lem-const}
Let $\varphi \in W^{s, p}(\Omega) \cap L^\infty(\Omega)$ be such that $\mathrm{supp}\,{\varphi} \Subset \Omega$ and $\varphi= 1$ in a neighborhood of the origin. If $u$ is a weak solution of \eqref{eq-main} in $\Omega \setminus \{0\}$, then the value
\begin{equation}\label{eq-const-Kbar}
\overline{K}\coloneqq \mathcal{E}(u, \varphi)+\int_{\Omega} b(x, u(x))(\varphi(x)-1) \,\mathrm{d}x
\end{equation}
is finite and independent of the particular choice of $\varphi$.
\end{lemma}

\begin{proof}
It follows from Lemma~\ref{lem-test-const} that $\mathcal{E}(u, \varphi)$ is finite. Moreover, \eqref{eq-b-Holder} shows that
\begin{equation*}
    \left|\int_{\Omega} b(x, u(x))(\varphi(x)-1)\,\mathrm{d}x \right| \leq \int_{\Omega \setminus N} (b_1|u|^{p-1}+b_2)|\varphi-1| \,\mathrm{d}x \leq C \left( \|u\|_{W^{s, p}(\Omega \setminus N)}^{p-1} +1 \right) < \infty,
\end{equation*}
where $N$ is a neighborhood of the origin on which $\varphi =1$, and $C$ is a constant depending only on $n$, $s$, $p$, $\varepsilon$, $\|b_1\|_{L^{n/(sp-\varepsilon)}(\Omega)}$, $\|b_2\|_{L^{n/(sp-\varepsilon)}(\Omega)}$ and $\Omega$.

Let $\varphi$ and $\overline{\varphi}$ be two functions satisfying the conditions given in the statement. It then follows from Proposition~\ref{prop-test} and Corollary~\ref{cor-sol} that
\begin{equation*}
\mathcal{E}(u, \varphi - \overline{\varphi})+\int_{\Omega} b(x, u(x))(\varphi(x)-\overline{\varphi}(x))\,\mathrm{d}x=0,
\end{equation*}
which implies the desired conclusion.
\end{proof}

We next provide a weaker version of upper bounds for singular solutions. For this purpose, we fix a ball $B_{4R}=B_{4R}(0) \Subset \Omega$ with $R>0$ to be determined later in the proof of Lemma~\ref{lem-positive-K} (see \eqref{eq-R}), depending only on $n$, $s$, $p$, $\Lambda$ and $b_1$, and define
\begin{equation*}
    \mathcal{B}(r) \coloneqq \int_{B_{4R} \setminus B_r} |b(x, u(x))|\,\mathrm{d}x \quad \text{for }0<r<4R.
\end{equation*}

\begin{lemma}\label{lem-upper}
Let $u$ be given as in Lemma~\ref{lem-non-rem}. Then
\begin{equation}\label{eq-upper-B}
u^{p-1}(x) \lesssim
\begin{cases}
(1+\mathcal{B}(|x|))|x|^{-(n-sp)} &\text{if}~sp<n, \\
(1+\mathcal{B}(|x|))(-\log |x|)^{p-1} &\text{if}~sp=n,
\end{cases}
\end{equation}
in a neighborhood of the origin.
\end{lemma}
In fact, \eqref{eq-upper-B} holds in $B_{4R}$ since $u$ is bounded away from the origin.

\begin{proof}
Assume first that $n \geq 2$. Set
\begin{equation}\label{eq-Mm}
M=M_r=\max_{\overline{B}_r \setminus B_{r/2}}u \quad\text{and}\quad m=m_r=\min_{\overline{B}_r \setminus B_{r/2}}u
\end{equation}
for $r \leq R$. Then $M \geq m >0$. It is enough to show that for sufficiently small $r\in(0,1)$
\begin{equation}\label{eq-M}
M^{p-1} \lesssim
\begin{cases}
(1+\mathcal{B}(r)) r^{-(n-sp)} &\text{if}~sp<n, \\
(1+\mathcal{B}(r)) (-\log r) &\text{if}~sp=n.
\end{cases}
\end{equation}
Let $\eta \in C^\infty_c(B_{3R})$ be a cut-off function such that $0 \leq \eta \leq 1$, $|\nabla \eta| \leq 2/R$ in $\mathbb{R}^n$ and $\eta = 1$ in $B_{2R}$. Define
\begin{equation*}
v=u\eta \quad \text{and} \quad \varphi=\min\{ v/m, 1\}.
\end{equation*}
Since $v/m \in W^{s, p}_{\mathrm{loc}}(\Omega \setminus \{0\})$ and $\lim_{x \to 0} v(x)/m=\infty$, it follows from Lemma~\ref{lem-test-trunc} that $\varphi \in W^{s, p}_{\mathrm{loc}}(\Omega)$. Moreover, since $\varphi=\varphi \bar{\eta}$ for any $\bar{\eta} \in C^\infty_c(\Omega)$ such that $\bar{\eta}= 1$ on $B_{3R}$, Lemma~\ref{lem-test-prod} shows that $\varphi \in W^{s, p}(\Omega)$. It is clear that $\varphi= 1$ in some neighborhood of the origin and that $\supp{\varphi} \Subset \Omega$, so $\overline{K}$ given by \eqref{eq-const-Kbar} is a constant independent of $\varphi$ by Lemma~\ref{lem-const}.

Since 
\begin{equation*}
m^{p-1}|\varphi(x)-\varphi(y)|^p \leq |v(x)-v(y)|^{p-2}(v(x)-v(y))(\varphi(x)-\varphi(y)),
\end{equation*}
we obtain that
\begin{equation}\label{eq-cutoff}
m^{p-1} \mathcal{E}(\varphi, \varphi) \leq \mathcal{E}(v, \varphi) \leq  |\mathcal{E}(u\eta, \varphi) - \mathcal{E}(u, \varphi)|+\overline{K}+\left|\int_{\Omega} b(x, u(x))(\varphi(x)-1) \,\mathrm{d}x \right|.
\end{equation}
We now estimate each term in the right-hand side of \eqref{eq-cutoff}. First of all, we see that
\begin{align}\label{eq-upper-I1234}
\begin{split}
|\mathcal{E}(u\eta, \varphi)-\mathcal{E}(u, \varphi)|
&\leq \int_{B_{4R} \setminus B_R} \int_{B_{4R} \setminus B_R} |\mathcal{A}(u\eta)-\mathcal{A}(u)| |\varphi(x)-\varphi(y)| k(x, y) \,\mathrm{d}y\,\mathrm{d}x \\
&\quad + 2\int_{B_{4R} \setminus B_R} \int_{\Omega \setminus B_{4R}} |\mathcal{A}(u\eta)-\mathcal{A}(u)| \varphi(x) k(x, y) \,\mathrm{d}y\,\mathrm{d}x \\
&\quad + 2\int_{B_R} \int_{\Omega \setminus B_R} |\mathcal{A}(u\eta)-\mathcal{A}(u)| |\varphi(x)-\varphi(y)| k(x, y) \,\mathrm{d}y\,\mathrm{d}x \\
&\quad + 2\int_\Omega \int_{\mathbb{R}^n \setminus \Omega} |\mathcal{A}(u\eta)-\mathcal{A}(u)| \varphi(x) k(x, y) \,\mathrm{d}y\,\mathrm{d}x \\
&\eqqcolon I_1 + I_2 + I_3 + I_4,
\end{split}
\end{align}
where $\mathcal{A}(u) = \mathcal{A}(u)(x, y)\coloneqq |u(x)-u(y)|^{p-2}(u(x)-u(y))$.

Since
\begin{align*}
|\mathcal{A}(u\eta)-\mathcal{A}(u)|^{\frac{p}{p-1}}
&\leq C |(u\eta)(x)-(u\eta)(y)|^p + C |u(x)-u(y)|^p \\
&\leq C |u(x)-u(y)|^p + C u^p(x) |\eta(x)-\eta(y)|^p
\end{align*}
and
\begin{align*}
|\varphi(x)-\varphi(y)|^p
&\leq m^{-p} |(u\eta)(x)-(u\eta)(y)|^p \\
&\leq Cm^{-p} (|u(x)-u(y)|^p + u^p(x) |\eta(x)-\eta(y)|^p),
\end{align*}
it follows from Young's inequality that
\begin{align}\label{eq-upper-I1}
\begin{split}
I_1
&\leq \frac{C}{m} \int_{B_{4R} \setminus B_R} \int_{B_{4R} \setminus B_R} \frac{|\mathcal{A}(u\eta)-\mathcal{A}(u)|^{\frac{p}{p-1}} + m^p |\varphi(x)-\varphi(y)|^p}{|x-y|^{n+sp}} \,\mathrm{d}y\,\mathrm{d}x \\
&\leq \frac{C}{m} [u]_{W^{s, p}(B_{4R} \setminus B_R)}^p + \frac{C}{m} R^{-sp} \|u\|_{L^p(B_{4R}\setminus B_R)}^p,
\end{split}
\end{align}
where $C=C(n, s, p, \Lambda)>0$. On the other hand, since
\begin{equation}\label{eq-A}
|\mathcal{A}(u\eta)-\mathcal{A}(u)| \leq C(p)(|u(x)|^{p-1} + |u(y)|^{p-1})
\end{equation}
and $|\varphi(x)-\varphi(y)| \leq 1$, we have by Lemma~\ref{lem-tail} that
\begin{align}\label{eq-upper-I2}
\begin{split}
I_2
&\leq C \int_{B_{3R} \setminus B_R} \int_{\Omega \setminus B_{4R}} \frac{u^{p-1}(x)+u^{p-1}(y)}{|x-y|^{n+sp}} \,\mathrm{d}y\,\mathrm{d}x \\
&\leq \frac{C}{R^{sp}} \left( \|u\|_{L^{p-1}(B_{3R})}^{p-1} + R^n \mathrm{Tail}^{p-1}(u; 0, R) \right),
\end{split}
\end{align}
where $C=C(n, s, p, \Lambda)>0$. Similar computations show that
\begin{equation}\label{eq-upper-I34}
\begin{split}
I_3 + I_4
&\leq C \int_{B_R} \int_{\Omega \setminus B_{2R}} \frac{u^{p-1}(x)+u^{p-1}(y)}{|x-y|^{n+sp}} \,\mathrm{d}y\,\mathrm{d}x + C \int_{B_{3R}} \int_{\mathbb{R}^n \setminus \Omega} \frac{u^{p-1}(x)+|u(y)|^{p-1}}{|x-y|^{n+sp}} \,\mathrm{d}y\,\mathrm{d}x \\
&\leq \frac{C}{R^{sp}} \left( \|u\|_{L^{p-1}(B_{3R})}^{p-1} + R^n \mathrm{Tail}^{p-1}(u; 0, R) \right)
\end{split}
\end{equation}
where $C=C(n, s, p, \Lambda)>0$.

Next, we estimate
\begin{equation}\label{eq-upper-J12}
\begin{split}
    \left|\int_{\Omega} b(x, u(x))(\varphi(x)-1) \,\mathrm{d}x \right|
    &\leq \int_{\Omega \cap \{v \leq m\}} |b(x, u(x))| \,\mathrm{d}x \\
    &\leq \mathcal{B}(r) + \int_{B_r \cap \{v \leq m\}} |b(x, u(x))| \,\mathrm{d}x + \int_{\Omega \setminus B_{4R}} |b(x, u(x))| \,\mathrm{d}x \\
    &\eqqcolon \mathcal{B}(r) + J_1 + J_2.
\end{split}
\end{equation}
By applying H\"older's inequality and using $r<1$, we can estimate $J_1$ as
\begin{equation}\label{eq-upper-J1}
        J_1 \leq \int_{B_r} (b_1m^{p-1}+b_2)\,\mathrm{d}x\leq Cr^{n-sp+\varepsilon}m^{p-1}+C,
\end{equation}
where $C>0$ depends only on $n$, $s$, $p$, $\varepsilon$, $\|b_1\|_{L^{n/(sp-\varepsilon)}(\Omega)}$ and $\|b_2\|_{L^{n/(sp-\varepsilon)}(\Omega)}$. Moreover, \eqref{eq-b-Holder} shows that
\begin{equation}\label{eq-upper-J2}
J_2 \leq \int_{\Omega \setminus B_{4R}} (b_1|u|^{p-1}+b_2) \,\mathrm{d}x \leq C \|u\|_{W^{s, p}(\Omega \setminus B_{4R})}^{p-1} + C,
\end{equation}
where $C$ depends only on $n$, $s$, $p$, $\varepsilon$, $\|b_1\|_{L^{n/(sp-\varepsilon)}(\Omega)}$, $\|b_2\|_{L^{n/(sp-\varepsilon)}(\Omega)}$ and $\Omega$.

By combining \eqref{eq-cutoff}--\eqref{eq-upper-I1} and \eqref{eq-upper-I2}--\eqref{eq-upper-J2}, we obtain that
\begin{equation}\label{eq-upper-E}
m^p \mathcal{E}(\varphi, \varphi) \leq C+(C+\overline{K}+\mathcal{B}(r))m+Cr^{n-sp+\varepsilon}m^p,
\end{equation}
where $C$ is a positive constant depending on $n$, $s$, $p$, $\Lambda$, $R$, $d$, $\eta$ and $u$. Since the function $\varphi$ is admissible for $\mathrm{cap}_{s, p}(\overline{B}_r \setminus B_{r/2}, B_{3R})$, it follows from Lemma~\ref{lem-cap-ball} that
\begin{equation}\label{eq-upper-cap}
\Lambda\mathcal{E}(\varphi, \varphi) \geq [\varphi]_{W^{s, p}(\mathbb{R}^n)}^p \geq \mathrm{cap}_{s, p}(\overline{B}_r \setminus B_{r/2}, B_{3R}) \eqsim
\begin{cases}
    r^{n-sp} &\text{if}~sp<n, \\
    (-\log r)^{1-p} &\text{if}~sp=n,
\end{cases}
\end{equation}
provided that $r$ is sufficiently small. Thus, by taking small $r \in (0,1)$, we obtain from \eqref{eq-upper-E} and \eqref{eq-upper-cap} that
\begin{equation*}
    m^p \lesssim \left( 1+(1+\mathcal{B}(r))m \right)
    \begin{cases}
    r^{n-sp} &\text{if}~sp<n, \\
    (-\log r)^{1-p} &\text{if}~sp=n.
\end{cases}
\end{equation*}
Note that Lemma~2 in Serrin~\cite{Ser64} implies that
\begin{equation}\label{eq-m}
m^{p-1} \lesssim
\begin{cases}
(1+\mathcal{B}(r))r^{-(n-sp)} &\text{if}~sp<n, \\
(1+\mathcal{B}(r))(-\log r)^{p-1} &\text{if}~sp=n.
\end{cases}
\end{equation}
Therefore, the desired upper bound \eqref{eq-M} for sufficiently small $r>0$ follows from \eqref{eq-m} and the Harnack inequality (Theorem~\ref{thm-Harnack-annulus}).

Next, we assume that $n=1$. In this case, the argument proceeds analogously to the case $n\geq2$, with $M$ and $m$ in \eqref{eq-Mm} now defined by
\begin{equation}\label{eq-Mm-pm}
M= M_r^\pm = \max_{I_r^{\pm}}u \quad\text{and}\quad m= m_r^{\pm}=\min_{I_r^{\pm}}u,
\end{equation}
where
\begin{equation}\label{eq-interval}
I_r^{+}=[r/2, r] \quad\text{and}\quad I_r^-=[-r, -r/2].
\end{equation}
The only difference is that the function $\varphi$ is now admissible for $\mathrm{cap}_{s, p}(I_r^{\pm}, (-3R, 3R))$ instead of $\mathrm{cap}_{s, p}(\overline{B}_r \setminus B_{r/2}, B_{3R})$. Nevertheless, by Lemma~\ref{lem-cap-ball}, we still obtain the lower bound
\begin{equation}\label{eq-upper-cap-1d}
\mathrm{cap}_{s, p}(I_r^{\pm}, (-3R, 3R)) \gtrsim 
\begin{cases}
    r^{n-sp} &\text{if}~sp<n, \\
    (-\log r)^{1-p} &\text{if}~sp=n,
\end{cases}
\end{equation}
which replaces \eqref{eq-upper-cap}. This leads to \eqref{eq-m} with the new $m=m_r^\pm$, and consequently \eqref{eq-M} with the new $M=M_r^\pm$ also follows from the Harnack inequality in Theorem~\ref{thm-harnack}.
\end{proof}

In order to improve the upper bound provided in Lemma~\ref{lem-upper}, we need to investigate the integral term $\mathcal{B}(r)$ in detail.

\begin{lemma}\label{lem-integrable}
Let $u$ be given as in Lemma~\ref{lem-non-rem}. Then 
    \begin{equation*}
        \int_{B_r} |b(x, u(x))| \,\mathrm{d}x \lesssim 
        \begin{cases}
            r^{\varepsilon} & \text{if $sp<n$},\\
            r^{\varepsilon/2} & \text{if $sp=n$}.
        \end{cases}
    \end{equation*}
In particular, $b(x, u(x)) \in L^1(B_{4R})$.
\end{lemma}

\begin{proof}
By Lemma~\ref{lem-non-rem}, there exists $r_1 \in (0, 2R)$ such that $u \geq 1$ in $B_{2r_1}$. By applying H\"older's inequality and Lemma~\ref{lem-upper}, we have for any $r < r_1$,
    \begin{equation*}
    \begin{split}
        \mathcal{B}(r)-\mathcal{B}(2r)
        &\leq \int_{B_{2r} \setminus B_r} (b_1+b_2)u^{p-1} \,\mathrm{d}x\\
        &\leq Cr^{n-sp+\varepsilon}\|b_1+b_2\|_{L^{n/(sp-\varepsilon)}(B_{2r})} \begin{cases}
(1+\mathcal{B}(r))r^{\varepsilon} &\text{if}~sp<n, \\
(1+\mathcal{B}(r))r^{\varepsilon/2} &\text{if}~sp=n.
\end{cases}
    \end{split}
    \end{equation*}
Therefore, the conclusion follows from the same argument provided in the proof of \cite[Lemma~2]{Ser65}.
\end{proof}

The upper bound of singular solutions is now an immediate consequence of Lemmas~\ref{lem-upper} and \ref{lem-integrable}, both in the subcritical and critical cases.
\begin{proof}[Proofs of the upper bound in Theorems~\ref{thm-iso-sub} and \ref{thm-iso-log}]
We may assume, by adding a suitable constant $C$ to $u$, that $u$ is positive in $\Omega \setminus \{0\}$. In this case, $v\coloneqq u+C$ satisfies $\mathcal{L}v+\bar{b}(x, v)=0$ in $\Omega \setminus \{0\}$, where $\bar{b}(x, z) \coloneqq b(x, z-C)$ enjoys the same structure condition as $b$, namely,
\begin{align*}
    \bar{b}(x, z)
    &\leq \max\{2^{p-2}, 1\}b_1(x) |z|^{p-1} + C^{p-1}\max\{2^{p-2}, 1\}b_1(x)+b_2(x) \\
    &\eqqcolon \bar{b}_1(x)|z|^{p-1} + \bar{b}_2(x).
\end{align*}
Therefore, we can apply Lemmas~\ref{lem-upper} and \ref{lem-integrable}, which leads us to the desired upper bound of singular solutions.
\end{proof}

In order to derive the lower bound of singular solutions in Theorem~\ref{thm-iso-sub}, we consider a new invariant quantity $K$ which is similar to $\overline{K}$.

\begin{lemma}\label{lem-K}
Let $\varphi \in W^{s, p}(\Omega) \cap L^\infty(\Omega)$ be such that $\mathrm{supp}\,{\varphi} \Subset \Omega$ and $\varphi= 1$ in a neighborhood of the origin. Let $u$ be given as in Lemma~\ref{lem-non-rem}. Then the value
\begin{equation}\label{eq-const-K}
K\coloneqq \mathcal{E}(u, \varphi)+\int_{\Omega} b(x, u(x))\varphi(x) \,\mathrm{d}x
\end{equation}
is finite and independent of the particular choice of $\varphi$.
\end{lemma}

\begin{proof}
    It immediately follows from the argument given in the proof of Lemma~\ref{lem-const} together with Lemma~\ref{lem-integrable}.
\end{proof}

The sign of $K$ can be determined by elaborating the estimates for $I_2$, $I_3$, $I_4$ and $\int_{\Omega}b \varphi$ appearing in the proof of Lemma~\ref{lem-upper}, with the help of the removability theorem (Corollary~\ref{cor-rem}) and the upper bound (Lemma~\ref{lem-upper}). We consider the case of $sp<n$, but we conjecture that Lemma~\ref{lem-positive-K} is also true when $sp=n$.

\begin{lemma}\label{lem-positive-K}
Assume that $sp<n$. Let $u$ be given as in Lemma~\ref{lem-non-rem}. Then the value $K$ given by \eqref{eq-const-K} is positive.
\end{lemma}

\begin{proof}
The upper bound in Theorem~\ref{thm-iso-sub} shows that
\begin{equation}\label{eq-upper}
u(x) \lesssim |x|^{-\tau} \quad\text{for all }x \in B_{R},
\end{equation}
where $\tau\coloneqq (n-sp)/(p-1)>0$. 

We suppose for contradiction that $K \leq 0$ and claim that there exist constants $A \geq 1$, $\delta \in (0,1)$ and $r_0 \in (0, R)$ such that
	\begin{equation}\label{eq-claim}
		m \leq A r^{-\tau(1-\delta)} \quad\text{for all}~r\in (0, r_0),
	\end{equation}
where $m=m_r$ is given as in \eqref{eq-Mm} if $n\geq 2$, and $m=m_r^\pm$ as in \eqref{eq-Mm-pm} if $n=1$.
Once we prove the claim \eqref{eq-claim}, then it follows from Theorem~\ref{thm-Harnack-annulus} that
 \begin{equation*}
\sup_{B_r \setminus B_{r/2}}{u} \leq C r^{-\tau(1-\delta)} + C \left(\frac{r}{R}\right)^{\frac{sp}{p-1}}\mathrm{Tail}(u_-; 0, R) + C\left(\frac{r}{R}\right)^{\frac{\varepsilon}{p-1}}\left( R^\varepsilon \|b_2\|_{L^{n/(sp-\varepsilon)}(B_R)} \right)^{\frac{1}{p-1}}
\end{equation*}
if $n\geq 2$, and from Theorem~\ref{thm-harnack} and the nonnegativity assumption on $u$ that
\begin{align*}
\sup_{I_r^\pm}{u}
&\leq C r^{-\tau(1-\delta)} + C \mathrm{Tail}(u_-; 3r/4, r/4) + C \left( r^\varepsilon \|b_2\|_{L^{n/(sp-\varepsilon)}(r/4, 5r/4)} \right)^{\frac{1}{p-1}} \\
&\leq C r^{-\tau(1-\delta)} + C \left( \frac{r}{R} \right)^{\frac{sp}{p-1}} \mathrm{Tail}(u_-; 0, R) + C \left( \frac{r}{R} \right)^{\frac{\varepsilon}{p-1}} \left( R^\varepsilon \|b_2\|_{L^{n/(sp-\varepsilon)}(0, 2R)} \right)^{\frac{1}{p-1}}
\end{align*}
if $n=1$, where $I_r^\pm$ is given in \eqref{eq-interval}.
	Hence, Corollary~\ref{cor-rem} shows that $u$ has a removable singularity at the origin, yielding a contradiction.

Let us now prove the claim \eqref{eq-claim}. Let $\eta$ and $\varphi$ be given as in the proof of Lemma~\ref{lem-upper}. By applying Lemma~\ref{lem-K} and repeating the estimates \eqref{eq-cutoff}--\eqref{eq-upper-I1}, we have
\begin{equation}\label{eq-pos-E}
\begin{aligned}
    m^{p-1} \mathcal{E}(\varphi, \varphi) &\leq |\mathcal{E}(u\eta, \varphi)-\mathcal{E}(u, \varphi)|+K+\left|\int_{\Omega} b(x, u(x))\varphi(x) \,\mathrm{d}x \right|\\
    &\leq \frac{C}{m} + C( I_2+I_3+I_4)+\left|\int_{B_{3R}} b(x, u(x))\varphi(x) \,\mathrm{d}x \right|.
\end{aligned}
\end{equation}
We first derive estimates for $I_2$, $I_3$ and $I_4$ sharper than \eqref{eq-upper-I2}--\eqref{eq-upper-I34}.
To this end, we define two sets
	\begin{equation*}
		T_1 \coloneqq \{u <m\} \cap B_r^c \quad \text{and} \quad T_2 \coloneqq \{u \geq m\} \cup B_r
	\end{equation*}
	so that $T_1 \cup T_2 = \mathbb{R}^n$.

We first use \eqref{eq-A} to see that
	\begin{equation*}
		\begin{aligned}
			I_2
			&\leq C\int_{T_1 \cap (B_{3R} \setminus B_R)} \int_{\Omega \setminus B_{4R}} \frac{u^{p-1}(x)+u^{p-1}(y)}{|x-y|^{n+sp}} \frac{u(x)}{m} \,\mathrm{d}y\,\mathrm{d}x\\
			&\quad+C\int_{T_2 \cap (B_{3R} \setminus B_R)} \int_{\Omega \setminus B_{4R}} \frac{u^{p-1}(x)+u^{p-1}(y)}{|x-y|^{n+sp}}  \,\mathrm{d}y\,\mathrm{d}x\\
			&\eqqcolon I_{2, 1}+I_{2, 2}.
		\end{aligned}
	\end{equation*}
It thus follows from $\|u\|_{L^\infty(B_{3R} \setminus B_R)} < \infty$ and \eqref{eq-upper-I2} that $I_{2,1} \leq C/m$. We postpone the estimate of $I_{2,2}$.

The estimates for $I_3$ and $I_4$ require the upper bound \eqref{eq-upper}. We first use \eqref{eq-A} and
\begin{equation*}
(u^{p-1}(x)+u^{p-1}(y))(u(x)+u(y)) \leq C(u^p(x)+u^p(y))
\end{equation*}
to see that
	\begin{equation*}
		\begin{aligned}
			I_3
			&\leq \frac{C}{m} \int_{T_1 \cap B_R} \int_{\Omega \setminus B_{2R}} \frac{u^p(x)+u^p(y)}{|x-y|^{n+sp}} \,\mathrm{d}y\,\mathrm{d}x\\
			&\quad+C\int_{T_2 \cap  B_R} \int_{\Omega \setminus B_{2R}} \frac{u^{p-1}(x)+u^{p-1}(y)}{|x-y|^{n+sp}}\left(1+\frac{u(y)}{m} \right)\,\mathrm{d}y\,\mathrm{d}x\\
			&\eqqcolon I_{3, 1}+I_{3, 2}.
		\end{aligned}
	\end{equation*}
Now by \eqref{eq-upper} and $\|u\|_{L^\infty(\Omega \setminus B_{2R})} < \infty$, it follows that
	\begin{equation*}
		\begin{aligned}
			I_{3, 1} \leq \frac{C}{m} \int_{B_R \setminus B_r} |x|^{-p\tau} \,\mathrm{d}x+\frac{C}{m} \leq \frac{C}{m} + \frac{C}{m}
			\begin{cases}
				1 &\text{if }n-p\tau > 0, \\
				\log (R/r) &\text{if } n-p\tau =0, \\
				r^{n-p\tau} &\text{if } n-p\tau<0.
			\end{cases}
		\end{aligned}
	\end{equation*}
Since $\log(R/r) \leq Cr^{-\kappa}$ for any $\kappa \in (0, \tau)$,
\begin{equation*}
I_{3, 1} \leq \frac{C}{m} \left( 1+r^{-\kappa}+r^{n-p\tau} \right).
\end{equation*}
Also, we have that
	\begin{align*}
			I_4
			&\leq C\int_{T_1 \cap B_{3R}} \int_{\mathbb{R}^n \setminus \Omega} \frac{u^{p-1}(x)+|u(y)|^{p-1}}{|x-y|^{n+sp}} \frac{u(x)}{m} \,\mathrm{d}y\,\mathrm{d}x\\
			&\quad+C\int_{T_2 \cap B_{3R}} \int_{\mathbb{R}^n \setminus \Omega} \frac{u^{p-1}(x)+|u(y)|^{p-1}}{|x-y|^{n+sp}}\,\mathrm{d}y\,\mathrm{d}x\\
			&\eqqcolon I_{4, 1}+I_{4, 2}
	\end{align*}
and that
	\begin{equation*}
			I_{4, 1} \leq \frac{C}{m} \int_{B_{3R} \setminus B_r} u^{p}(x)\,\mathrm{d}x+\frac{C}{m} \frac{\mathrm{Tail}^{p-1}(u; 0, R)}{R^{sp}} \int_{B_{3R} \setminus B_r} u(x) \,\mathrm{d}x.
	\end{equation*}
	Since $u \leq 1+u^p$, we obtain as in the estimate of $I_{3, 1}$ that
\begin{equation*}
I_{4, 1} \leq \frac{C}{m} \left( 1+r^{-\kappa}+r^{n-p\tau} \right).
\end{equation*}
	
Let us now focus on the terms $I_{2, 2}$, $I_{3, 2}$ and $I_{4, 2}$. Since $m \to \infty$ as $r \to 0$ by Lemma~\ref{lem-non-rem}, there exists sufficiently small $r_1 \in (0, R)$ such that $\|u\|_{L^\infty(\Omega \setminus B_{R})} <m$ if $r<r_1$. In particular, we have $T_2 \cap B_{3R} \subset T_2 \cap B_{R}$ when $r <r_1$, which shows that $I_{2, 2}=0$. Moreover, the choice of $r_1$ implies that
\begin{equation*}
I_{3, 2} + I_{4, 2} \leq C\int_{T_2 \cap B_{R}} \int_{\mathbb{R}^n \setminus B_{2R}} \frac{u^{p-1}(x)+|u(y)|^{p-1}}{|x-y|^{n+sp}}\,\mathrm{d}y\,\mathrm{d}x.
\end{equation*}
We only consider those values of $r>0$ for which $r^{-\tau(1-\delta)}<m$ since otherwise \eqref{eq-claim} holds trivially. Let $\theta>0$ be a constant that will be determined soon. The upper bound \eqref{eq-upper} of $u$ shows that for any $x \in B_{R} \setminus B_{\theta r^{1-\delta}}$
	\begin{equation*}
		u(x) \leq C|x|^{-\tau} \leq C \theta^{-\tau} r^{-\tau(1-\delta)} <C\theta^{-\tau}m<m,
	\end{equation*}
	provided that $\theta$ is chosen sufficiently large (independent of $r$). This implies that
	\begin{equation*}
		T_2 \cap B_{R} \subset B_r \cup B_{\theta r^{1-\delta}}.
	\end{equation*}
We now choose $r_0 \in (0, r_1)$ so that $r \leq \theta r^{1-\delta}$ for all $r \in (0, r_0)$ which leads to
	\begin{equation*}
		T_2 \cap B_{R} \subset B_{\theta r^{1-\delta}}.
	\end{equation*}
It thus follows from Lemma~\ref{lem-tail}, \eqref{eq-upper} and $\varepsilon<sp<n$ that
\begin{align*}
I_{3, 2} + I_{4, 2}
&\leq C\int_{B_{\theta r^{1-\delta}}} \int_{\mathbb{R}^n \setminus B_{2R}} \frac{u^{p-1}(x)+|u(y)|^{p-1}}{|x-y|^{n+sp}}\,\mathrm{d}y\,\mathrm{d}x \\
&\leq C \left( \int_{B_{\theta r^{1-\delta}}} u^{p-1} \,\mathrm{d}x + r^{(1-\delta)n} \mathrm{Tail}^{p-1}(u; 0, R) \right) \\
&\leq C r^{(1-\delta) sp} \leq C r^{(1-\delta) \varepsilon}.
\end{align*}

We next split the last integral term in \eqref{eq-pos-E} as
\begin{equation*}
    \left|\int_{B_{3R}} b\varphi\right| \leq \int_{T_1 \cap (B_{3R} \setminus B_{R})}|b|\varphi\,\mathrm{d}x+\int_{T_1 \cap B_{R}}|b|\varphi\,\mathrm{d}x+\int_{T_2 \cap B_{3R}}|b|\varphi\,\mathrm{d}x \eqqcolon J_1+J_2+J_3.
\end{equation*}
By employing a similar argument as before, we obtain that
\begin{equation*}
        J_1 \leq \frac{1}{m} \int_{T_1 \cap (B_{3R}\setminus B_{R})} (b_1u^{p-1}+b_2)u\,\mathrm{d}x \leq \frac{C}{m}.
\end{equation*}
Moreover, by following the proof of \eqref{eq-b-Holder}, using the fact that $u=m\varphi$ in $T_1 \cap B_R$ and assuming that $R<1$, we have
\begin{equation*}
    \begin{aligned}
        J_2 
        &\leq \int_{B_R} (b_1 (m\varphi)^{p-1}+b_2)\varphi\,\mathrm{d}x\\
        &\leq m^{p-1} |B_{R}|^{\frac{\varepsilon}{n}}\|b_1\|_{L^{\frac{n}{sp-\varepsilon}}(B_R)} \|\varphi\|_{L^{p_s^{\ast}}(B_{3R})}^{p}+|B_{R}|^{1-\frac{sp-\varepsilon}{n}-\frac{1}{p_s^{\ast}}}\|b_2\|_{L^{\frac{n}{sp-\varepsilon}}(B_R)}\|\varphi\|_{L^{p_s^{\ast}}(B_{3R})}\\
        &\leq Cm^{p-1} \|b_1\|_{L^{\frac{n}{sp-\varepsilon}}(B_R)} [\varphi]_{W^{s, p}(\mathbb{R}^n)}^{p}+C\|b_2\|_{L^{\frac{n}{sp-\varepsilon}}(B_R)}[\varphi]_{W^{s, p}(\mathbb{R}^n)},
    \end{aligned}
\end{equation*}
where $C=C(n, s, p, \varepsilon)>0$. We now choose $R$ sufficiently small so that
\begin{equation}\label{eq-R}
C\|b_1\|_{L^{n/(sp-\varepsilon)}(B_R)}\leq \frac{1}{4\Lambda}.
\end{equation}
Then, after an application of Young's inequality, we arrive at
\begin{equation*}
    J_2 \leq \frac{1}{2}m^{p-1}\mathcal{E}(\varphi, \varphi)+\frac{C}{m},
\end{equation*}
where $C=C(n, s, p, \Lambda, \varepsilon, \|b_2\|_{L^{n/(sp-\varepsilon)}(B_R)})>0$.

Finally, since $T_2 \cap B_{3R} \subset T_2 \cap B_{R} \subset B_{\theta r^{1-\delta}}$, we utilize Lemma~\ref{lem-integrable} to obtain
\begin{equation*}
    J_{3} \leq \int_{B_{\theta r^{1-\delta}}} |b(x, u(x))|\,\mathrm{d}x \leq Cr^{(1-\delta)\varepsilon}.
\end{equation*}
Therefore, combining all these estimates concludes that
	\begin{equation*}
		m^{p}\mathcal{E}(\varphi, \varphi) \lesssim 1+r^{-\kappa}+r^{n-p\tau}+mr^{(1-\delta)\varepsilon}.
	\end{equation*}
	By recalling \eqref{eq-upper-cap} for $n\geq 2$ and \eqref{eq-upper-cap-1d} for $n=1$, and applying 	Lemma~2 in Serrin~\cite{Ser64}, we have 
	\begin{equation*}
		m \lesssim r^{-\frac{n-sp+\kappa}{p}} + r^{-\tau+s} + r^{-\tau+\frac{(1-\delta)\varepsilon}{p-1}}.
	\end{equation*}
	We now choose a constant $\delta \in (0,1)$ small enough so that
	\begin{equation*}
		-\tau (1-\delta) \leq \min\left\{ -\frac{n-sp+\kappa}{p},\, -\tau+s,\, -\tau+\frac{(1-\delta)\varepsilon}{p-1}\right\},
	\end{equation*}
	which proves the claim \eqref{eq-claim}.
\end{proof}

The positivity of $K$ now allows us to prove the lower bound of singular solutions.

\begin{lemma}\label{lem-lower}
Assume that $sp<n$. Let $u$ be given as in Lemma~\ref{lem-non-rem}. Then
\begin{equation*}
u \gtrsim |x|^{-(n-sp)/(p-1)}
\end{equation*}
in a neighborhood of the origin.
\end{lemma}

\begin{proof}
The upper bound in Theorem~\ref{thm-iso-sub} shows that \eqref{eq-upper} holds. Let $\delta \in (0, 1)$ be a constant to be determined later. For $r < \delta R$ we set 
\begin{equation*}
\widetilde{M}=\widetilde{M}_{\delta r}=\max_{\overline{B}_R \setminus B_{\delta r}}u.
\end{equation*}
Let $\phi$ be the $(-\Delta_p)^s$-potential of $\overline{B}_r$ in $B_{\delta R}$, as defined in Kim--Lee--Lee~\cite[Definition~2.15]{KLL23}, then $\phi$ satisfies the conditions (for $\varphi$) in Lemma~\ref{lem-K}. Thus, Lemmas~\ref{lem-K} and \ref{lem-positive-K} show that the value $K$ defined by \eqref{eq-const-K} is a positive constant. We write
\begin{align*}
K
	&= \iint_{(B_{R}^c \times B_{R}^c)^c \setminus (B_{\delta r} \times B_{\delta r})} |u(x)-u(y)|^{p-2}(u(x)-u(y))(\phi(x)-\phi(y)) k(x, y) \,\mathrm{d}y\,\mathrm{d}x \\
	&=\int_{B_{R} \setminus B_{\delta r}} \int_{B_{R} \setminus B_{\delta r}} |u(x)-u(y)|^{p-2}(u(x)-u(y))(\phi(x)-\phi(y)) k(x, y) \,\mathrm{d}y\,\mathrm{d}x \\
	&\quad +2\int_{B_R \setminus B_{\delta r}} \int_{\mathbb{R}^n \setminus B_{R}} |u(x)-u(y)|^{p-2}(u(x)-u(y))\phi(x) k(x, y) \,\mathrm{d}y\,\mathrm{d}x \\
	&\quad +2\int_{B_{\delta r}} \int_{\mathbb{R}^n \setminus B_{\delta r}} |u(x)-u(y)|^{p-2}(u(x)-u(y))(1-\phi(y)) k(x, y) \,\mathrm{d}y\,\mathrm{d}x \\
    &\quad +\int_{\Omega}b(x, u(x))\phi(x)\,\mathrm{d}x\\
	&\eqqcolon I_1+I_2+I_3+J.
\end{align*}
We first note that by Lemma~\ref{lem-integrable}, there exists a constant $\delta_1 \in (0, 1)$ such that 
\begin{equation*}
    |J|\leq \int_{B_{\delta R}} |b(x, u(x))|\,\mathrm{d}x \leq C(\delta R)^{\varepsilon} \leq \frac{K}{8}
\end{equation*}
for all $\delta \in (0, \delta_1]$.

We next let $\kappa > 0$. An application of Young's inequality shows that
\begin{equation*}
I_1 \leq \frac{\kappa}{\widetilde{M}} \int_{B_{R} \setminus B_{\delta r}} \int_{B_{R} \setminus B_{\delta r}} |u(x)-u(y)|^p k(x, y) \,\mathrm{d}y\,\mathrm{d}x + C \widetilde{M}^{p-1} \mathcal{E}(\phi, \phi)
\end{equation*}
for some $C=C(p, \kappa)>0$. Let $\eta \in C^\infty_c(\Omega)$ be such that $\eta = 1$ in $B_{2R}$ and $0 \leq \eta \leq 1$ in $\mathbb{R}^n$. If we define
\begin{equation*}
v=u\eta \quad\text{and}\quad \varphi=\min\{u\eta/\widetilde{M}, 1 \},
\end{equation*}
then the same argument as in the proof of Lemma~\ref{lem-upper} shows that $\varphi \in W^{s, p}(\Omega)$, $\supp \varphi \Subset \Omega$ and $\varphi = 1$ in a neighborhood of the origin. Note also that $v=u$ and $\varphi=u/\widetilde{M}$ in $B_R \setminus B_{\delta r}$. It thus follows that
\begin{align*}
&\frac{1}{\widetilde{M}} \int_{B_{R} \setminus B_{\delta r}} \int_{B_{R} \setminus B_{\delta r}} |u(x)-u(y)|^p k(x, y) \,\mathrm{d}y\,\mathrm{d}x \\
&=\int_{B_{R} \setminus B_{\delta r}} \int_{B_{R} \setminus B_{\delta r}} |v(x)-v(y)|^{p-2}(v(x)-v(y))(\varphi(x)-\varphi(y)) k(x, y) \,\mathrm{d}y\,\mathrm{d}x.
\end{align*}
Since $|v(x)-v(y)|^{p-2}(v(x)-v(y))(\varphi(x)-\varphi(y)) \geq \widetilde{M}^{p-1}|\varphi(x)-\varphi(y)|^p \geq 0$ for all $x, y \in \mathbb{R}^n$, we obtain that
\begin{align*}
I_1
&\leq \kappa \mathcal{E}(v, \varphi) + C(\kappa) \widetilde{M}^{p-1} \mathcal{E}(\phi, \phi) \\
&\leq \kappa K + \kappa |\mathcal{E}(u\eta, \varphi) - \mathcal{E}(u, \varphi)| + C(\kappa) \widetilde{M}^{p-1} \mathcal{E}(\phi, \phi).
\end{align*}
By arguing as in the proof of Lemma~\ref{lem-upper}, we deduce that
\begin{align*}
I_1 \leq \kappa \left( K + C + C/\widetilde{M} \right) + C(\kappa) \widetilde{M}^{p-1} \mathcal{E}(\phi, \phi).
\end{align*}
Since $\widetilde{M}>1$, we can make $\kappa (K+C+C/\widetilde{M}) \leq \kappa (K+2C) \leq K/8$ by taking $\kappa>0$ sufficiently small. It thus follows that
\begin{equation*}
I_1 \leq \frac{K}{8} + C\widetilde{M}^{p-1} \mathcal{E}(\phi, \phi).
\end{equation*}
For $I_2$, we use Lemma~\ref{lem-tail} to obtain that
\begin{align*}
I_2
&= 2\int_{B_{\delta R} \setminus B_{\delta r}} \int_{\mathbb{R}^n \setminus B_{R}} |u(x)-u(y)|^{p-2}(u(x)-u(y))\phi(x) k(x, y) \,\mathrm{d}y\,\mathrm{d}x \\
&\leq C \int_{B_{\delta R} \setminus B_{\delta r}} \int_{\mathbb{R}^n \setminus B_{R}} \frac{u^{p-1}(x)+|u(y)|^{p-1}}{|x-y|^{n+sp}} \,\mathrm{d}y\,\mathrm{d}x \\
&\leq \frac{C}{R^{sp}} \int_{B_{\delta R}} u^{p-1} \,\mathrm{d}x + C \frac{(\delta R)^n}{R^{sp}} \mathrm{Tail}^{p-1}(u; 0, R/2).
\end{align*}
Moreover, the pointwise upper bound \eqref{eq-upper} of $u$ shows that
	\begin{equation*}
		\int_{B_{\delta R}}u^{p-1} \,\mathrm{d}x \leq C(\delta R)^{sp}.
	\end{equation*}
Thus, there exists $\delta_2=\delta_2(n,s,p, u) \in (0,1)$ such that
\begin{equation*}
I_2 \leq K/8
\end{equation*}
for all $\delta \in (0, \delta_2]$.

Since $u>0$ in $B_{R}$,
\begin{align*}
	I_3
	&\leq 2\Lambda \int_{B_{\delta r}} \int_{\mathbb{R}^n \setminus B_{r}} \frac{(u(x)-u(y))_+^{p-1} (1-\phi(y))}{|x-y|^{n+sp}} \,\mathrm{d}y\,\mathrm{d}x\\
	&\leq 2\Lambda \int_{B_{\delta r}} \int_{\mathbb{R}^n \setminus B_{R}} \frac{(u(x)-u(y))^{p-1}_+}{|x-y|^{n+sp}} \,\mathrm{d}y\,\mathrm{d}x+2\Lambda \int_{B_{\delta r}} \int_{B_R \setminus B_{r}} \frac{u^{p-1}(x)}{|x-y|^{n+sp}} \,\mathrm{d}y\,\mathrm{d}x.
\end{align*}
By using Lemma~\ref{lem-tail} and \eqref{eq-upper} again, we obtain that
\begin{equation*}
I_3 \leq C\delta^{sp} +C(\delta r)^nR^{-sp} \mathrm{Tail}^{p-1}(u; 0, R/2) \leq C \delta^{sp}.
\end{equation*}
Thus, there exists $\delta_3=\delta_3(n,s,p, u) \in (0,1)$ such that
\begin{equation*}
I_3 \leq K/8
\end{equation*}
for all $\delta \in (0, \delta_3]$.

We now fix $\delta=\min\{\delta_1, \delta_2, \delta_3\} \in (0,1)$. By combining the estimates for $I_1$, $I_2$, $I_3$ and $J$, we arrive at
\begin{equation*}
K \leq C\widetilde{M}^{p-1} \mathcal{E}(\phi, \phi) \leq C\widetilde{M}^{p-1} [\phi]_{W^{s, p}(\mathbb{R}^n)}^p.
\end{equation*}
Since $\phi$ is the $(-\Delta_p)^s$-potential of $\overline{B}_r$ in $B_{\delta R}$, Lemma~2.16 (iii) in Kim--Lee--Lee~\cite{KLL23} and Lemma~\ref{lem-cap-ball} show that
\begin{equation*}
[\phi]_{W^{s, p}(\mathbb{R}^n)}^p = \mathrm{cap}_{s, p}(\overline{B}_r, B_{\delta R}) \eqsim \left(r^{-\frac{n-sp}{p-1}}-R^{-\frac{n-sp}{p-1}} \right)^{1-p} \lesssim r^{n-sp},
\end{equation*}
provided that $r>0$ is small enough. Therefore, we arrive at
\begin{equation}\label{eq-tildeM}
	\widetilde{M} \gtrsim r^{-(n-sp)/(p-1)}.
\end{equation}

It remains to transfer this information to the behavior of $u$ in an annulus whose inner and outer radii are comparable to $r$. We first assume that $n \geq 2$. We choose $\rho \in [\delta r, R]$ so that
\begin{equation*}
	\widetilde{M}=\max_{\partial B_{\rho}}u.
\end{equation*}
Then by \eqref{eq-tildeM} and \eqref{eq-upper}, we see that
\begin{equation*}
	c_1 r^{-(n-sp)/(p-1)} \leq \widetilde{M} \leq c_2 \rho^{-(n-sp)/(p-1)},
\end{equation*}
which implies that
\begin{equation*}
	\rho \leq (c_1/c_2)^{-(p-1)/(n-sp)} r\eqqcolon \overline{c}r.
\end{equation*}
Therefore, we conclude that
\begin{equation*}
	\widetilde{M}=\max_{\overline{B}_{\overline{c}r} \setminus B_{\delta r}}u,
\end{equation*}
and so the Harnack inequality (Theorem~\ref{thm-Harnack-annulus}) finishes the proof.

Next, we assume that $n=1$. In this case, we have
\begin{equation*}
    \widetilde{M} = \max_{[\delta r, R]} u \quad\text{or}\quad \widetilde{M} = \max_{[-R, -\delta r]} u.
\end{equation*}
Without loss of generality, we may assume that $\widetilde{M} = \max_{[\delta r, R]} u$. Then, by repeating the proof for the case $n\geq 2$ and using the Harnack inequality in Theorem~\ref{thm-harnack} instead of Theorem~\ref{thm-Harnack-annulus}, we conclude that
\begin{equation*}
    u \gtrsim |x|^{-(1-sp)/(p-1)}
\end{equation*}
in $[\delta r, \bar{c} r]$ for some $\bar{c}>0$. Moreover, the weak Harnack inequality developed in Lemma~\ref{lem-WHI} shows that
\begin{equation*}
	\left(\fint_{B_{\bar{c} r}} u^{p_0}\,\mathrm{d}x\right)^{\frac{1}{p_0}} \leq C \inf_{B_{\bar{c} r}} u+C \mathrm{Tail}(u_-; 0, R) + C\left( R^\varepsilon \|b_2\|_{L^{\frac{n}{sp-\varepsilon}}(B_R)} \right)^{\frac{1}{p-1}}.
\end{equation*}
The lower bound in the positive interval $[\delta r, \bar{c}r]$ is then transferred to the desired bound in the negative interval $[-\bar{c}r, -\delta r]$.
\end{proof}

\begin{proof}[Proof of the lower bound in Theorem~\ref{thm-iso-sub}]
As in the proof of the upper bound, we may assume that $u$ is positive in $\Omega \setminus \{0\}$. Thus, the desired result follows from Lemma~\ref{lem-lower}.
\end{proof}

\end{document}